\documentclass{amsart}
\usepackage{amssymb}
\usepackage{bbm}

\usepackage{hyperref}
\usepackage{cite}

\newtheorem{theorem}{Theorem}
\newtheorem{remark}[theorem]{Remark}
\newtheorem{lemma}[theorem]{Lemma}

\newtheorem{prop}[theorem]{Proposition}

\newtheorem{notation}[theorem]{Notation}

\numberwithin{equation}{section}
\numberwithin{theorem}{section}

\DeclareMathOperator{\arccosh}{arccosh}

\email{jjc3jr@virginia.edu}
\thanks{\footnotemark {$^*$} Both authors were supported in part by NSF DMS 1255574.}
\email{melcher@virginia.edu}

\begin{document}

\title[Hermite functions on odd-dimensional spheres]{Small-time asymptotics for subelliptic Hermite functions on $SU(2)$ and the CR sphere}
\author[Campbell]{Joshua Campbell{$^*$}}
\address[Joshua Campbell]{Department of Mathematics\\
University of Virginia\\
Charlottesville, Virginia 22904 USA}
\author[Melcher]{Tai Melcher{$^*$}}
\address[Tai Melcher]{Department of Mathematics\\
University of Virginia\\
Charlottesville, Virginia 22904 USA}

\begin{abstract}
	We show that, under a natural scaling, the small-time behavior of the logarithmic derivatives of the subelliptic heat kernel on $SU(2)$ converges to their analogues on the Heisenberg group at time 1. Realizing $SU(2)$ as $\mathbb{S}^3$, we then generalize these results to higher-order odd-dimensional spheres equipped with their natural subRiemannian structure, where the limiting spaces are now the higher-dimensional Heisenberg groups.
\end{abstract}

\maketitle
\tableofcontents

\section{Introduction}

Heat kernels are a classical object of study and are known to have deep relations to the topological and geometric properties of the space on which they live. In particular, small-time asymptotics of the heat kernel on a Riemannian manifold can reveal various geometric data about the underlying space. The logarithmic derivatives of the heat kernel, or rather, the Hermite functions, generalize the Hermite polynomials on $\mathbb{R}^d$. Of course, the Hermite polynomials play a key role in the study of the heat kernel on $\mathbb{R}^d$, but also show up in many other parts of analysis; thus Hermite functions are a natural object of interest.

In \cite{Mitchell1999} Mitchell studied the small-time behavior of Hermite functions on compact Lie groups. In particular, he showed that, when written in exponential coordinates with a natural re-scaling, these functions converge to the classical Euclidean Hermite polynomials. Later in \cite{Mitchell2001}, Mitchell showed that  Hermite functions on compact Riemannian manifolds, again written in exponential coordinates with appropriate re-scaling, admit asymptotic expansions in powers of $\sqrt{t}$, with a classical Hermite polynomial as the leading coefficient (and the other coefficients are other polynomials). The present paper is concerned  with heat kernels related to the natural subRiemannian structure on $SU(2)\cong\mathbb{S}^3$ and, more generally, on the CR sphere $\mathbb{S}^{2d+1}$. The aim is to show that a statement analogous to \cite{Mitchell1999} holds for the Hermite functions on the CR sphere,  where the limiting objects are now the subRiemannian Hermite functions of the Heisenberg group. 

The results of \cite{Mitchell1999,Mitchell2001} may be interpreted as a strong quantification of how compact Riemannian manifolds are locally Euclidean, and the present paper may be viewed as an extension of those results in a particular subRiemannian setting. In general, the tangent cone approximation of a subRiemannian geometry by the appropriate stratified group (its nilpotentization) is much weaker than the tangent space approximation of a Riemannian manifold by its Euclidean tangent space. Thus such extensions should not be taken for granted.

Viewed from another perspective, there has been recent growing interest in various functional inequalities in the context of CR geometry, motivated in part by the isoperimetric inequality and a large class of spectral problems, all of which are deeply tied to the geometry of the underlying space. Recent work has exploited the relationship between the CR sphere and the Heisenberg group, either directly via the Cayley transform equivalence or more generally thinking of $\mathbb{S}^{2d+1}$ as a ``Heisenberg manifold'', to give a variety of results on the CR sphere. For example, there has been remarkable recent success in computing optimal Sobolev-type inequalities \cite{BFM2013,FrankLieb2012}. This relationship has also been used to explore spectral multipliers \cite{CKS2011} and determination of the eigenvalues \cite{2013Aribi} of the sub-Laplacian and variations on the Webster scalar curvature problem \cite{2002Malchiodi,2013Cao} on the CR sphere. Given the potential geometric information that small-time asymptotics yield, and the relative ease of analysis in the stratified Lie group setting, results like those in the present paper are of significant interest.

\subsection{Statement of Results}
\label{s.statements}

We begin our study by focusing on the special case of $SU(2)$, which is a (real) compact Lie group. This case most closely parallels the results appearing in \cite{Mitchell1999}, but even as the simplest case we consider, still demonstrates all the necessary elements of the proof. In fact, it turns out that, given the symmetries of the heat kernel for higher-dimensional spheres, the analysis of the $SU(2)\cong\mathbb{S}^3$ kernel is mostly sufficient to complete the proof for the more general case on $\mathbb{S}^{2d+1}$.

\subsubsection{The Lie group case}
\label{s.su(2)}

\begin{notation}
	Given a Lie group $G$, for $g\in G$ let $\ell_g$ denote left translation by $g$.  For any element $\xi$ of the Lie algebra $\mathfrak{g}=\mathrm{Lie}(G)$, let $\tilde{\xi}$ denote the associated left-invariant vector field; that is, $\tilde\xi(g) = \ell_{g*}\xi$ for all $g\in G$ and $\tilde{\xi}$ acts on smooth functions $f\in C^\infty(G)$ by 
\[ \tilde{\xi}(g)f = (\tilde{\xi} f)(g) = \frac{d}{d\varepsilon}\bigg|_0 f(ge^{\varepsilon\xi}). \]
For any $\xi=\xi_1\otimes\cdots\otimes\xi_k\in\mathfrak{g}^{\otimes k}$, let $\tilde{\xi}$ denote the $k^\mathrm{th}$-order left-invariant differential operator $\tilde{\xi}_1\cdots\tilde{\xi}_k$. 
\end{notation}

As usual, $SU(2)$ denotes the $2\times2$ complex unitary matrices of determinant 1
\begin{equation}
\label{e.su2-1}
\begin{split} 
SU(2) &= \{ A\in M_2(\mathbb{C}) : A^*A = I \text{ and }  \det A=1 \}\\
	&= \Bigg\{ \begin{pmatrix} 
		a & b\\
		-\bar{b} & \bar{a} \end{pmatrix}: a,b\in\mathbb{C}^2, |a|^2+|b|^2=1 \Bigg\} \end{split}
\end{equation}
which has 
Lie algebra  $\mathfrak{su}(2)$, the $2\times2$ complex skew-adjoint trace-free matrices. The Pauli matrices
\[ X = \begin{pmatrix}
0 & 1 \\
-1 & 0 \\
 \end{pmatrix} , 
Y = \begin{pmatrix}
0 & i \\
i & 0 \\
 \end{pmatrix}, \text{ and } 
Z = \begin{pmatrix}
i & 0 \\
0 & -i \\
 \end{pmatrix} \]
form a basis of $\mathfrak{su}(2)$ and satisfy the commutation relations
\begin{equation}
\label{e.su2} 
[X,Y]=2Z, [Y,Z]=2X, \text{ and } [Z,X]=2Y. 
\end{equation}
We will use the cylindrical coordinates for $SU(2)$ introduced in \cite{CowlingSikora2001},
\[ (r,\theta, z) \mapsto \exp(r\cos\theta X+r\sin\theta Y) \exp(zZ) = \begin{pmatrix}
 e^{i z}\cos r &  e^{i(\theta - z)}\sin r \\
- e^{-i(\theta - z)}\sin r &  e^{-i z}\cos r \\
 \end{pmatrix} \]
for $r \in [0,\frac{\pi}{2})$, $\theta \in [0, 2\pi]$, and $z \in [-\pi,\pi]$.
Since (\ref{e.su2}) implies that $\{X,Y\}$ generates the span of $\mathfrak{su}(2)$, we say that $\{X,Y\}$ satisfies {\it H\"ormander's condition} on $\mathfrak{su}(2)$, and thus the operator $L = \tilde{X}^2 + \tilde{Y}^2$ is subelliptic. We let $p:(0,\infty)\times SU(2)\times SU(2)\rightarrow\mathbb{R}$ denote the fundamental solution of the heat equation $\partial_t p_t(g,g') = L_g p_t(g,g')$ and $\lim_{t\downarrow 0} p_t(g,g')= \delta_{g'}(g)$. As usual, we let $p_t(g):=p_t(e,g)$ where $e$ is the identity element of $SU(2)$ (that is, the $2\times2$ identity matrix).

The Heisenberg group $\mathbb{H}$ also plays an important role in the sequel. Recall that $\mathbb{H}$ may be realized as $\mathbb{R}^3$ equipped with the group operation
\[ (x,y,z)\cdot(x',y',z') = (x+x',y+y',z+z'+ xy'-yx'); \]
the Lie algebra $\mathfrak{h}$ has basis $\{\mathcal{X},\mathcal{Y},\mathcal{Z}\}$ satisfying the bracket relations
\[ [\mathcal{X},\mathcal{Y}]=2\mathcal{Z} \text{ and } [\mathcal{X}, \mathcal{Z}] = [\mathcal{Y},\mathcal{Z}] = 0. \]
As above, $\{\mathcal{X},\mathcal{Y}\}$ forms a H\"ormander set, and thus the operator $\mathcal{L} = \tilde{\mathcal{X}}^2 + \tilde{\mathcal{Y}}^2$ is subelliptic. We let $h:(0,\infty)\times \mathbb{H}\times \mathbb{H}\rightarrow \mathbb{R}$ denote the fundamental solution of the heat equation for $\mathcal{L}$, and again we take $h_t(g)=h_t(e,g)$ where the identity on the Heisenberg group is 0.

That $L$ and $\mathcal{L}$ are subelliptic implies that $p$ and $h$ are smooth. In particular, for any $k\in\mathbb{N}$ and $\xi=\xi_1\otimes\cdots\otimes\xi_k\in\mathfrak{su}(2)^{\otimes k}$, we can define the Hermite function corresponding to $\xi$ as
\[ K_\xi(t,g) := \frac{(\tilde{\xi} p_t)(g)}{p_t(g)}.\] 
For $\xi=\xi_1\otimes\cdots\otimes\xi_k\in\{X,Y,Z\}^{\otimes k}$  we define 
\[ |\xi|  := \#\{i: \xi_i=X \text{ or } Y\} + 2\#\{i: \xi_i=Z\}. \]
Similarly, for any $\eta=\eta_1\otimes\cdots\otimes\eta_k\in\mathfrak{h}^{\otimes k}$, let the Hermite function corresponding to $\eta$ be
\[ H_\eta(t,g) := \frac{(\tilde{\eta} h_t)(g)}{h_t(g)}.\] 
Finally, we define a mapping $\beta$ from the basis elements of $\mathfrak{su}(2)$ to those of $\mathfrak{h}$ by $\beta:X\mapsto\mathcal{X},Y\mapsto\mathcal{Y},Z\mapsto\mathcal{Z}$ and extend it to $\{X,Y,Z\}^{\otimes k}$ by
\[ \{X,Y,Z\}^{\otimes k}\ni \xi=\xi_1\otimes\cdots\otimes\xi_k
	\mapsto \beta(\xi_1)\otimes\cdots\otimes\beta(\xi_k)=:\beta(\xi). \]

\begin{theorem}
\label{t.main}
Let $\xi=\xi_1\otimes\cdots\otimes\xi_k\in\mathfrak{su}(2)^{\otimes k}$ with $\xi_i\in\{X,Y,Z\}$. Then 
\[ \lim_{t\downarrow0} t^{|\xi|/2} K_\xi(t,(\sqrt{t}r,\theta,tz)) = H_{\beta(\xi)}(1,(r,\theta,z)) \]
uniformly on compact subsets of $[0,\infty)\times[0,2\pi)\times\mathbb{R}$.
\end{theorem}

	For comparison, the scaling used in the model results of Mitchell \cite{Mitchell1999} is uniform in the spatial coordinates, which is appropriate for the Riemannian setting. The scaling that we employ is natural given the dilation structure on the Heisenberg group, and the fact that dilation of $SU(2)$ yields the Heisenberg group (for example, see \cite{Ricci1986} and the generalization to spheres in \cite{DooleyGupta1999}). Also, while Mitchell is able to state his results generally for arbitrary left-invariant differential operators of rank $k$ (that is, $\xi$ a homogeneous element of $\mathfrak{g}^{\otimes k}$), it's necessary in the subelliptic case to adapt the scaling in $t$ to the dilation structure as well. Of course, we could have stated the above result for slightly more general left-invariant differential operators by considering for example homogeneous elements of $(\mathrm{span}\{X,Y\})^{\otimes k}$ and $(\mathrm{span}\{Z\})^{\otimes k}$ and defining an appropriate mapping $\beta$ to $(\mathrm{span}\{\mathcal{X},\mathcal{Y}\})^{\otimes k}$ and $(\mathrm{span}\{\mathcal{Z}\})^{\otimes k}$. But this is implicit in the stated result.
	
	Now that we understand more precisely the type of result of interest, we generalize the statement to odd-dimensional spheres $\mathbb{S}^{2d+1}$.

\subsubsection{The CR sphere}
\label{s.intro-sphere}

Let $\mathbb{S}^{2d+1}$ denote the boundary of the unit ball in $\mathbb{C}^{d+1}$. In standard coordinates, this is
\[ \mathbb{S}^{2d+1} := \{(z_1,\ldots,z_{d+1})\in\mathbb{C}^{d+1}: |z_1|^2+\cdots+|z_{d+1}|^2=1\}. \]
Let $\mathcal{S}:= \sum_{j=1}^{d+1} z_j\frac{\partial}{\partial {z_j}}$, and 
define the vector fields 
\[ T_j := \frac{\partial}{\partial {z_j}} - \overline{z}_j\mathcal{S} \]
for $j=1,\ldots,d+1$; then the $T_j$ generate the span of the holomorphic tangent space $T_{1,0}\mathbb{S}^{2d+1}=T_{1,0}\mathbb{C}^{d+1}\cap\mathbb{C}T\mathbb{S}^{2d+1}$.
The subLaplacian on $\mathbb{S}^{2d+1}$ is given by
\[  L := 2\sum_{j=1}^{d+1} (T_j\overline{T}_j+\overline{T}_jT_j); \]
$L$ is essentially self-adjoint on $C^\infty(\mathbb{S}^{2d+1})$ with respect to the uniform measure. The transversal direction is the real vector field
\[ T_0 := i(\mathcal{S}-\overline{\mathcal{S}}) = i\sum_{j=1}^{d+1} \left(z_j\frac{\partial}{\partial {z_j}} - \overline{z}_j \frac{\partial}{\partial{\overline{z}_j}}\right), \]
and $\mathbb{C}T\mathbb{S}^{2d+1}$ is generated by $T_j,\overline{T}_j,T_0$.

We now introduce the cylindrical coordinates on $\mathbb{S}^{2d+1}$ used in \cite{BaudoinWang2013}. 
As usual, $\mathbb{CP}^{d}\subset\mathbb{C}^{d+1}\setminus\{0\}$ will denote complex projective space, the set of one-dimensional complex-linear subspaces of $\mathbb{C}^{d+1}$; that is, $\mathbb{CP}^{d}\cong(\mathbb{C}^{d+1}\setminus\{0\})/ \sim$ where
\[ (w_1,\ldots,w_d)\sim (\lambda w_1,\ldots,\lambda w_d) \text{ for all } \lambda\in\mathbb{C}\setminus\{0\}. \]
$\mathbb{CP}^{d}$ is a complex manifold of (complex) dimension $d$, 
and one can construct an atlas on $\mathbb{CP}^{d}$ via the patches
\[ U_k = \{(z_1,\ldots,z_{d+1}):z_k\ne0\} \]
for $k=1,\ldots, d+1$ by 
\[ w^k=(w^k_1,\ldots,w^k_{k-1},w^k_{k+1},\ldots,w^k_{d+1})
	=\left(\frac{z_1}{z_k},\cdots,\frac{z_{k-1}}{z_k},\frac{z_{k+1}}{z_k},\cdots,\frac{z_{d+1}}{z_k}\right).\] 
In particular, since we focus on Brownian motion emitted from the north pole, we fix the local coordinates $w_j:=w_j^{d+1}=z_j/z_{d+1}$ for $j=1,\ldots,d$. The $z_j$'s are called the {\em homogeneous coordinates} and the $w_j$'s are the {\em inhomogeneous coordinates}.

Let $(w_1,\ldots,w_d,z)$ be local coordinates for $\mathbb{S}^{2d+1}$ where $(w_1,\ldots,w_d)$ are the local inhomogeneous coordinates for $\mathbb{CP}^{d}$ described above 
and $z$ is the local fiber coordinate, that is, $(w_1,\ldots,w_d)$ parameterizes the complex lines passing through the north pole and $z$ determines a point on the line that is unit distance from the north pole. More explicitly, these coordinates are given by
\[ (w_1,\ldots,w_d, z)\mapsto\left(\frac{w_1e^{iz}}{\sqrt{1+\rho^2}},\ldots,\frac{w_de^{i z}}{\sqrt{1+\rho^2}},\frac{e^{i z}}{\sqrt{1+\rho^2}}\right)\]
where $ z\in\mathbb{R}/2\pi\mathbb{Z}$, $w=(w_1,\ldots,w_d)\in\mathbb{CP}^d$, and $\rho=|w|=\sqrt{\sum_{j=1}^d |w_j|^2}$.

Let $\mathbb{H}^{2d+1}$ denote the real Heisenberg group of dimension $2d+1$. That is, let
\[ \mathfrak{h}^{2d+1} := \mathrm{span}\{\mathcal{X}_1,\mathcal{Y}_1,\ldots,\mathcal{X}_d,\mathcal{Y}_d,\mathcal{Z}_0\} \]
be equipped with the Lie bracket given by
\[ [\mathcal{X}_j,\mathcal{Y}_k] = 2\delta_{jk}\mathcal{Z}_0 \]
for all $j,k=1,\ldots,d$, and all other brackets not determined by anti-symmetry are 0. 
Then 
\[ \tilde{\mathcal{X}}_j = \frac{\partial}{\partial{x_j}} - y_j\frac{\partial}{\partial z}, 
	\tilde{\mathcal{Y}}_j =\frac{\partial}{\partial{y_j}} + x_j\frac{\partial}{\partial z}, \text{ and }
	\tilde{\mathcal{Z}}_0 = \frac{\partial}{\partial z}.  \]
are the associated left-invariant vector fields with respect to the group operation given by the Baker-Campbell-Hausdorff-Dynkin formula
\[ (\mathbf{x},\mathbf{y},z)\cdot(\mathbf{x}',\mathbf{y}',z')
	= \left(\mathbf{x}+\mathbf{x}',\mathbf{y}+\mathbf{y}',z + z' + \sum_{j=1}^d (x_jy_j'-x_j'y_j) \right)
\]
for $\mathbf{x},\mathbf{x}',\mathbf{y},\mathbf{y}'\in\mathbb{R}^d$ and $z,z'\in\mathbb{R}$. 
We let $\mathbb{H}^{2d+1}$ denote $\mathbb{R}^{2d+1}$ as a Lie group equipped with this operation, and the subLaplacian on $\mathbb{H}^{2d+1}$ is given by
\[ \mathcal{L} = \sum_{j=1}^d (\tilde{\mathcal{X}}_j^2+\tilde{\mathcal{Y}}_j^2). \]

To more easily see the parallel structure with $\mathbb{S}^{2d+1}$, for $j=1,\ldots,d$ we may take the coordinates $w_j=y_j+ix_j$ on $\mathbb{H}^{2d+1}$ so that
\[ \frac{\partial}{\partial w_j} = \frac{1}{2}\left(\frac{\partial}{\partial y_j}
	-i\frac{\partial}{\partial x_j}\right) 
\text{ and }
\frac{\partial}{\partial \overline{w}_j} = \frac{1}{2}\left(\frac{\partial}{\partial y_j}
	+i\frac{\partial}{\partial x_j}\right). \]
Then, if we let
\begin{align*}
\mathcal{\mathcal{Z}}_j &:= \frac{1}{2}(\tilde{\mathcal{Y}}_j-i\tilde{\mathcal{X}}_j)
	= \frac{\partial}{\partial w_j} + \frac{1}{2}i \overline{w}_j\frac{\partial}{\partial z}
\end{align*}
and
\[ \overline{\mathcal{Z}}_j := \frac{\partial}{\partial \overline{w}_j} - \frac{1}{2}i w_j\frac{\partial}{\partial z} \]
for $j=1,\ldots,d$, we may write 
\[ \mathcal{L}
	= 2\sum_{j=1}^d (\mathcal{Z}_j\overline{\mathcal{Z}}_j + \overline{\mathcal{Z}}_j\mathcal{Z}_j). \]

As in the Lie group setting, we consider the Hermite functions for $\mathbb{S}^{2d+1}$ and $\mathbb{H}^{2d+1}$, that is, the logarithmic derivatives of the heat kernels associated to the given subLaplacians. Let $p_{t,d}$ denote the fundamental solution to the Cauchy equation for $L$ on $\mathbb{S}^{2d+1}$ emitted from the ``north pole'', that is, the point $(0,\ldots,0,1)$. Analogously, $h_{t,d}$ will denote the fundamental solution to the Cauchy equations for $\mathcal{L}$ on $\mathbb{H}^{2d+1}$ emitted from the identity 0. 

For $\kappa=(\kappa_1,\ldots,\kappa_k)\in\{0,1,\ldots,d\}^k$, let 
\[ |\kappa| := \#\{i: \kappa_i\in\{1,\ldots,d\}\} + 2\#\{i:\kappa_i=0\} \]
and $T_\kappa := T_{\kappa_1}\cdots T_{\kappa_k}$ and similarly for $\mathcal{Z}_\kappa$.

\begin{theorem}
\label{t.sphere}
For $\kappa=(\kappa_1,\ldots,\kappa_k)\in\{0,1,\ldots,d\}^k$, we have that
\[ \lim_{t\downarrow0} t^{|\kappa|/2}\frac{(T_\kappa p_{t,d})(\sqrt{t}w_1,\ldots,\sqrt{t}w_d,t z)}{ p_{t,d}(\sqrt{t}w_1,\ldots,\sqrt{t}w_d,t z)} =
	\frac{(\mathcal{Z}_\kappa h_{1,d})(w_1,\ldots,w_d, z)}{h_{1,d}(w_1,\ldots,w_d, z)}  \]
uniformly on compact subsets of $\mathbb{CP}^d\times\mathbb{R}$.
\end{theorem}

\subsection{Discussion}

The proofs of Theorems \ref{t.main} and \ref{t.sphere} are based on a careful analysis of explicit expressions of the subelliptic heat kernel proved by Baudoin and coauthors in \cite{BaudoinBonnefont2009} and \cite{BaudoinWang2013}. Other derivations of the subelliptic kernel on $SU(2)\cong \mathbb{S}^3$ and more generally $\mathbb{S}^{2d+1}$ are available for example in \cite{CDFI2011} and \cite{Greiner2013}. The approach of \cite{BaudoinBonnefont2009} and \cite{BaudoinWang2013} is to take advantage of the contact structure of these manifolds to write the subLaplacian as $L=\Delta-Z^2$, where $\Delta$ is the standard Laplacian and $Z$ is the Reeb vector field which commutes with the subLaplacian. This allows the subelliptic heat kernel to be written as an integral transform of the Riemannian heat kernel, which has a well-known explicit form. In \cite{BaudoinBonnefont2009} and \cite{BaudoinWang2013}, the authors use these integral expressions to prove a variety of results about the subelliptic kernel, including results on small-time asymptotics. Relevant to the present paper, in Proposition 3.13 of \cite{BaudoinBonnefont2009}, the authors prove a ``zero-th order'' version of Theorem \ref{t.main}; that is, they prove that the subelliptic heat kernel on $SU(2)$ with the given scaling converges to (a constant multiple of) the subelliptic heat kernel on $\mathbb{H}^3$.  

We note that in \cite{Bonnefont2012, Wang2016} the authors use similar techniques to develop integral expressions for subelliptic heat kernels on anti-de Sitter spaces $\mathbf{H}^{2d+1}$ and their universal covers $\widetilde{\mathbf{H}^{2d+1}}$. In particular these expressions involve the same functions as those appearing in the $\mathbb{S}^{2d+1}$ kernel, reflecting in part their symmetry with the spaces considered here. For example, in the case $d=1$, we have $\mathbf{H}^3\cong SL(2,\mathbb{R})$ where
\begin{align*} 
SL(2,\mathbb{R}) &= \{ A\in M_2(\mathbb{R}) : \det A=1 \}
	\cong \{ (a,b)\in\mathbb{C} : |a|^2-|b|^2=1 \}
\end{align*}
is the Lie group with Lie algebra $\mathfrak{sl}(2) = \{ A\in M_2(\mathbb{R}) : \mathrm{tr}(A)=0 \}$. Note that $\mathfrak{sl}(2) = \mathrm{span}\{X,Y,Z\}$ where the basis elements satisfy the commutation relations
\[[X,Y]=2Z,\quad [Z,X]=-2Y,\quad [Y,Z]=-2X. \]
Comparing these with (\ref{e.su2-1}) and (\ref{e.su2}) shows the symmetries between $SU(2)$ and $SL(2,\mathbb{R})$. In particular, these similarities and the symmetric expressions for the subelliptic heat kernels allows us to extend the analysis and subsequent results of the present paper to those settings, although we do not do so explicitly in this paper.

In \cite{Mitchell1999}, which gives our model results in the elliptic setting on compact Lie groups, Mitchell uses a more general approach analyzing parametrix approximations to the kernels and their derivatives. Such analysis is more difficult in the subelliptic setting, but may still be possible, especially in the ``step two'' case. The present paper presents results in the $SU(2)$ setting first, and then treats $\mathbb{S}^{2d+1}$ essentially as a generalization of the three-dimensional case. There are other natural generalizations of the $SU(2)$ case that one could consider. In particular, any compact semisimple Lie group $G$ may be equipped with a natural ``step two'' subRiemannian structure by decomposing its Lie algebra into a Cartan subalgebra and its orthogonal complement with respect to the Killing form. The orthogonal complement then generates a subelliptic structure on $G$. Analogues of the results in the present paper in this setting are the subject of ongoing work.

In addition to the references already cited above, small-time asymptotics of subelliptic kernels have generated a great amount of work, see for example \cite{Barilari2013,BBN2012,BenArous1988,BenArous1989,BenArous1991,Franchi2014,Leandre1987-2,Leandre1987-1,SeguinMansouri2012} and their references.
We give further mention here to some results that are of particular relevance in the current setting. Beals, Greiner, and Stanton \cite{BGS1984} studied the small-time asymptotics of subelliptic heat kernels on CR manifolds by using pseudo-differential calculus. Beals, Gaveau, and Greiner \cite{BGG2000} obtained small-time estimates for the subelliptic heat kernel on the Heisenberg group through the explicit expression for the kernel developed independently by Gaveau \cite{Gaveau1977} and Hulanicki \cite{Hulanicki1976}. Using similar techniques to those developed in \cite{Bonnefont2012, Wang2016}, Baudoin and Wang study subelliptic heat kernels for the quaternionic Hopf fibration, including the small-time behavior, in \cite{BaudoinWang2014}.

\textit{Acknowledgement. } We thank the anonymous referee for valuable feedback that improved the paper, and in particular for identifying an error in an earlier version.

\section{Analysis of the $SU(2)$ elliptic heat kernel and its derivatives}
\label{s.qt}

Using the notation set in Section \ref{s.su(2)} of the introduction, let $\Delta=\Delta_{SU(2)}=X^2+Y^2+Z^2$ denote the standard Laplacian on $SU(2)$. The Riemannian heat kernel on $SU(2)$ is well-known; via the Poisson summation formula it may be written as, for $t>0$ and $g\in SU(2)$,
\[
	q_t(g) = q_t(\cos \delta) 
	= \frac{\sqrt{\pi}e^t}{4t^{3/2}}\frac{1}{\sin\delta}
		\sum_{k\in\mathbb{Z}} (\delta+2k\pi) e^{-(\delta+2k\pi)^2/4t} .
\]
where $\delta$ is the Riemannian distance of $g$ from the identity. Note that this is the kernel with respect to the normalized Haar measure $(\sin 2r)/4\pi^2\,dr\,d\theta\,dz$. When $g=(r,\theta,z)$ in the given cylindrical coordinates, we have $\cos\delta= \cos r\cos z$.
Reordering terms, we may rewrite this as
\begin{multline*} 
q_t(\cos \delta) = \frac{\sqrt{\pi}e^t}{4t^{3/2}}\frac{\delta}{\sin\delta}e^{-\delta^2/4t}\\
	\times \left( 1 + 2 \sum_{k=1}^\infty e^{-\pi^2k^2 /t} \left(\cosh\frac{\pi k\delta}{t} + \frac{2\pi k}{\delta} \sinh\frac{\pi k\delta}{t} \right) \right).
\end{multline*} 
It is clear that $q_t(\cos\delta)$ admits an analytic extension for $\delta\in\mathbb{C}$, and in particular 
\begin{align*}
q_t(\cosh\delta) 
	&= q_t(\cos i\delta) \\
	& = \frac{\sqrt{\pi}e^t}{4t^{3/2}} \frac{\delta}{\sinh\delta}e^{\delta^2/4t}\left(1+2\sum_{k=1}^\infty e^{-k^2\pi^2/t}\left(\cos\frac{\pi k \delta}{t}+\frac{2\pi k}{\delta}\sin\frac{\pi k \delta}{t}\right)\right).
\end{align*}
So we may write
\begin{equation}
	\label{e.qtx}
q_t(x) =  \frac{\sqrt{\pi}e^t}{4t^{3/2}} \cdot\left\{ \begin{array}{ll}
	Q_1(t,x) \left(1+ R_1(t,x) \right) & \text{for } x=\cos\delta\in[0,1] \\
	Q_2(t,x) \left(1+ R_2(t,x) \right) & \text{for } x=\cosh\delta\ge1
	\end{array} \right. 
\end{equation}
where
\[ Q_1(t,x) := \frac{\arccos x}{\sqrt{1-x^2}}e^{-(\arccos x)^2/4t},\]
\[ R_1(t,x):= 2\sum_{k=1}^\infty e^{-\pi^2k^2/t}\left(\cosh\left(\frac{\pi k \arccos x}{t}\right)+\frac{2\pi k}{\arccos x}\sinh\left(\frac{\pi k \arccos x}{t}\right)\right), \]
\[ Q_2(t,x):= \frac{\arccosh x}{\sqrt{x^2-1}}e^{(\arccosh x)^2/4t}, \]
and
\[ R_2(t,x) := 2\sum_{k=1}^\infty e^{-\pi^2k^2/t}\left(\cos\left(\frac{\pi k \arccosh x}{t}\right)+\frac{2\pi k}{\arccosh x}\sin\left(\frac{\pi k \arccosh x}{t}\right)\right). \]

In Proposition 3.5 of \cite{BaudoinBonnefont2009}, the authors take advantage of the facts that $L=\Delta-\partial_z^2$ and that $\Delta$ and $\partial_z$ commute to show that the subelliptic heat kernel on $SU(2)$ has the form (in cylindrical coordinates)
\[ p_t(r,z) = \frac{1}{\sqrt{4\pi t}}\int^\infty_{-\infty} e^{-(\lambda+i z)^2/4t} q_t(\cos r \cosh \lambda) d\lambda \]
for $r\in[0,\pi/2)$ and $z\in[-\pi,\pi]$.

\begin{remark}
Note that as expected the subelliptic kernel depends only the radial component of the horizontal coordinates (and the vertical coordinate $z$). The radial dependence reflects the ellipticity of the subLaplacian in the generating directions $X$ and $Y$.
\end{remark}

To analyze the derivatives of $p_t$, in the sequel we will first develop uniform bounds for derivatives of all orders of $Q_1$, $Q_2$, $R_1$ and $R_2$ that are necessary for the subsequent analysis. First let us set some basic notation.

\begin{notation}
\label{n.multi}
For any multi-index $\alpha=(\alpha_1,\ldots,\alpha_n)\in\{0,1,\ldots\}^n$ we set 
\[ |\alpha| := \sum_{k=1}^n \alpha_k \qquad \text{ and } \qquad 
	\|\alpha\|_n := \sum_{k=1}^n k\alpha_k. \]
We will also use $|\alpha|_{even}$ and $|\alpha|_{odd}$ for the sum of components with even and odd index respectively; that is,
\[ |\alpha|_{even} := \sum_{k=1}^{\lfloor n/2\rfloor} \alpha_{2k} \qquad \text{ and } \qquad
	|\alpha|_{odd} := \sum_{k=1}^{\lceil n/2\rceil} \alpha_{2k-1}. \]
Let
\[ \mathcal{J}_n := \{\alpha=(\alpha_1,\ldots,\alpha_n) : \|\alpha\|_n = n \} \]
with the convention that $\mathcal{J}_0$ is the empty set.
\end{notation}

\begin{remark}
\label{r.m}
One may note that $|(0,\ldots,0,1)|=1$ and $|(n,\ldots,0,0)|=n$, and for any other $\alpha\in\mathcal{J}_n$, $1<|\alpha|< n$. Also, for any $\alpha\in\mathcal{J}_n$,
\begin{align*}
|\alpha| = |\alpha|_{odd} + |\alpha|_{even} 
	\le |\alpha|_{odd} + 2|\alpha|_{even} 
	\le \|\alpha\|_n = n 
\end{align*}
with equality only when $\alpha=(\alpha_1,\alpha_2,0,\ldots,0)$, that is, when $\alpha_j=0$ for all $j=3,\cdots,n$.

\end{remark}

Using Notation \ref{n.multi}, we recall the standard chain rule formula
\begin{equation}
\label{chain} 
\frac{d^n}{dx^n} f(g(x)) = \sum_{\alpha\in\mathcal{J}_n} 
	\frac{n!}{\alpha_1!1!^{\alpha_1}\cdots \alpha_n!n!^{\alpha_n}} f^{(|\alpha|)}(g(x))
		\prod_{j=1}^n \left(g^{(j)}(x)\right)^{\alpha_j}. 
\end{equation}
In the sequel, for $n\in\mathbb{N}$ and $\alpha\in\mathcal{J}_n$ we will let
\[ c_{\alpha n} := \frac{n!}{\alpha_1!1!^{\alpha_1}\cdots \alpha_n!n!^{\alpha_n}}. \]

\subsection{Some useful limits and estimates}
In this section we establish some preliminary results that will allow uniform bounds on the derivatives of $q_t$.

\begin{notation}
\label{n.constant}
In the sequel, $C_{a_1,\ldots,a_r}$ represents a finite positive constant depending only the parameters $a_1,\ldots,a_r$. In a sequence of estimates or results, constants can and will vary from line to line with no distinction in notation.
\end{notation}

\begin{lemma}
	\label{l.newnew1}
	The functions $\arccos^2 x$ and $\arccosh^2 x$ are $C^\infty$ with bounded derivatives of all orders on $[0,1]$ and $[1,\infty)$, respectively.
\end{lemma}

\begin{proof}
Note that
	\[ \frac{d}{dx}\arccos^2 x  
	= \frac{2\arccos x}{\sqrt{1-x^2}} 
	\]
	which one may see is bounded on $[0,1)$ and approaches 2 as $x\uparrow1$.	

One may show more generally for $n\ge2$ that the derivatives of $\arccos^2 x$ are of the form 
	\[
\frac{d^n}{dx^n}\arccos^2 x  
	= \frac{M_n(x)}{(1-x^2)^{(2n-1)/2}}
	= \frac{p_n(x)\sqrt{1-x^2}+q_n(x)\arccos x}{(1-x^2)^{(2n-1)/2}} 
\]
where $p_n$ and $q_n$ are polynomials of degree $n-2$ and $n-1$, respectively. Clearly these derivatives are bounded for $x$ away from 1. We will now show by induction on $n$ that
	\begin{equation}\label{e.acos.lim}
\lim_{x\uparrow 1} \frac{d^n}{d x^n} \arccos^2 x = 2\frac{((n-1)!)^2}{(2n-1)!!},
	\end{equation}
which is evidently true in the case $n=1$. 
	First, we show by induction that 
	\begin{equation}
	\label{e.mn} M_n'=(n-1)^2M_{n-1}.\end{equation} 
Necessarily,
\begin{equation*}
M_{n+1} = M_n'(1-x^2) + M_n(2n-1)x;
\end{equation*}
	combining this with the induction hypothesis on (\ref{e.mn}) gives
	\[ M_{n+1}= (n-1)^2M_{n-1}(1-x^2) + M_n(2n-1)x;
	\]
and thus
\begin{align*}
M_{n+1}' &= (n-1)^2(M'_{n-1}(1-x^2)-M_{n-1}2x) + M_n'(2n-1)x + M_n(2n-1) \\
	&= (n-1)^2(M'_{n-1}(1-x^2)- M_{n-1}2x) \\\	
	&\qquad + (n-1)^2M_{n-1}(2n-1)x + M_n(2n-1) \\
	&= (n-1)^2(M'_{n-1}(1-x^2) + (2n-3)M_{n-1}x) + M_n(2n-1) \\
	&= (n-1)^2M_n + M_n(2n-1) 
	= n^2M_n,
\end{align*}
	thus verifying (\ref{e.mn}).
Now, an application of L'H\^{o}pital's gives that
	\begin{multline*}
\lim_{x\downarrow 1} \frac{d^n}{dx^n} \arccos^2 x 
	 = \lim_{x\downarrow 1} \frac{M_n(x)}{(x^2-1)^{(2n-1)/2}} \\
	 =\lim_{x\downarrow 1} \frac{(n-1)^2 M_{n-1}(x)}{(2n-1)x(x^2-1)^{(2n-3)/2}} 
	= \frac{(n-1)^2}{2n-1} \lim_{x\downarrow 1} \frac{d^{n-1}}{d x^{n-1}} \arccos^2 x.
\end{multline*}
	Combined with the induction hypothesis, this proves (\ref{e.acos.lim}).

The argument for $\arccosh^2 x$ is completely analogous. We may show that the derivatives of are of the form 
\begin{equation}
\label{e.jarccosh2}
\frac{d^n}{dx^n}\arccosh^2 x  
	= \frac{M_n(x)}{(x^2-1)^{(2n-1)/2}}
	= \frac{p_n(x)\sqrt{x^2-1}+q_n(x)\arccosh x}{(x^2-1)^{(2n-1)/2}} 
\end{equation}
where $p_1\equiv0$ and otherwise $p_n$ and $q_n$ are polynomials of degree $n-2$ and $n-1$, respectively. Thus, the derivatives decay to 0 as $x$ goes to infinity; for $x$ close to 1, we note that in this case the numerators satisfy
$M_n'=-(n-1)^2M_{n-1}$, and thus
	\[
	\lim_{x\downarrow 1} \frac{d^n}{d x^n} \arccosh^2 x = 2(-1)^n\frac{((n-1)!)^2}{(2n-1)!!}.
\]
\end{proof}

\begin{lemma}
\label{l.1-new}
Fix $K\ge1$. For all $n\in\mathbb{N}$ and $x\in(0,1)$
	\begin{align*}
\left| \frac{d^n}{dx^n} \cosh(K\arccos x)\right|
	&\le C_nK^ne^{K\pi/2}.
\end{align*} 
\end{lemma}

\begin{proof}
	Let $F(K,x) := \cosh(K \arccos x)$. We will show that $F^{(n)}(K,x)$ is negative and increasing on $(0,1)$ when $n$ is odd, and positive and decreasing on $(0,1)$ when $n$ is even. Thus the function is dominated on $(0,1)$ by its value at 0, from which the estimate follows. 
	
First, one may show that the derivatives have the form
\begin{align*}
F^{(n)}(K,x)
	&= \frac{N_n(K,x)}{(1-x^2)^{(2n-1)/2}} \\
	&= \frac{g_{n,c}(K,x)\sqrt{1-x^2}\cosh(K \arccos x) +g_{n,s}(K,x)\sinh(K \arccos x)}{(1-x^2)^{(2n-1)/2}}
\end{align*}
where $g_{n,c}$ and $g_{n,s}$ are polynomials in $K$ and $x$.  	
In particular, $N_1(K,x)= -K\sinh(K\arccos x)$ is negative and increasing on $(0,1)$, and 
\[ N_2(K,x) =K^2\sqrt{1-x^2}\cosh(K\arccos x) - Kx\sinh(K\arccos x) \] 
is positive and decreasing on $(0,1)$.
	Completely analogously to the proof of (\ref{e.mn}) in Lemma \ref{l.newnew1}, one may show by induction that $N_{n+1}'=(K^2+n^2)N_n$.
	Thus, assuming that $N_n$ is negative (and increasing) if $n$ is odd, and positive (and decreasing) if $n$ is even, this shows that $N_{n+1}$ monotone on $(0,1)$ as desired (decreasing if $n+1$ is even and increasing if $n+1$ is odd). Furthermore, based on their form, $N_{n+1}(K,1)=0$ and can only change sign at $x=1$. Thus, $N_{n+1}$ is positive for even $n+1$ and negative for odd $n+1$.
	
	Since the sign of $F^{(n)}$ is determined by $N_n$ and $\frac{d}{dx}F^{(n)}=F^{(n+1)}$, this gives the desired behavior of $F^{(n)}$. Finally, since $N_n'=(K^2+(n-1)^2)N_{n-1}$ and the definition of $N_n$ imply that
	\[ N_{n+1}= (K^2+(n-1)^2)N_{n-1}(1-x^2) + N_n(2n-1)x,
	\]
it also follows by induction that $g_{2n,c}$ and $g_{2n+1,c}$ are of degree $2n$ in $K$ and $g_{2n-1,s}$ and $g_{2n,s}$ are of degree $2n-1$ in $K$ giving the correct order in $K$.
\end{proof}

\begin{lemma}
\label{l.2-new}
	Fix $K\ge1$. For all $n\in\mathbb{N}$ and $x\in(0,1)$,
\[ \left|\frac{d^n}{dx^n} \frac{\sinh(K \arccos x)}{\arccos x}\right|
	\le C_n K^{n+1}e^{K\pi/2}. \]
\end{lemma}

\begin{proof}
First noting that
\[ \frac{\sinh(K \arccos x)}{\arccos x}
	= \sum_{m=0}^\infty \frac{K^{2m+1}}{(2m+1)!}(\arccos x)^{2m}, \]
we would like to say 
\begin{equation}
\label{e.sinhK}
\frac{d^{n}}{dx^{n}} \frac{\sinh(K \arccos x)}{\arccos x}
	= \sum_{m=0}^\infty \frac{K^{2m+1}}{(2m+1)!} 
		\frac{d^{n}}{dx^{n}}(\arccos x)^{2m}.
\end{equation}
To this end, using the notation from (\ref{chain}), we have
\begin{multline}
	\label{e.acos2m}
\frac{d^{n}}{dx^{n}} (\arccos x)^{2m}= \frac{d^{n}}{dx^{n}} (\arccos^2 x)^{m} \\
	= \sum_{\alpha\in\mathcal{J}_n, |\alpha|\le m} c_{ \alpha n} m(m-1)\cdots 
		(m-|\alpha|+1)(\arccos^2 x)^{m-|\alpha|} 
		\prod_{j=1}^n \left(\frac{d^{j}}{dx^{j}} \arccos^2 x \right)^{\alpha_j}.
\end{multline}
Combining this with Lemma \ref{l.newnew1} we have that for all $x\in(0,1)$
\begin{align*}
\left|\frac{d^{n}}{dx^{n}} (\arccos x)^{2m}\right|
	\le C_n \left(\frac{\pi}{2}\right)^{2m} m(m-1)\cdots((m-n)^++1).
\end{align*}
Since 
	\[ \sum_{m}^\infty \frac{K^{2m+1}}{(2m+1)!}\cdot C_n \left(\frac{\pi}{2}\right)^{2m} m^n<\infty, \]
this justifies (\ref{e.sinhK}).
Thus,
\begin{align*}
	\left|\frac{d^{n}}{dx^{n}} \frac{\sinh(K \arccos x)}{\arccos x}\right|
	&\le K\sum_{m=0}^\infty \frac{K^{2m}}{(2m)!}
		\left|\frac{d^{n}}{dx^{n}}(\arccos x)^{2m}\right| \\
	&\le K\sum_{m=0}^\infty \frac{K^{2m}}{(2m)!}
		C_n \left(\frac{\pi}{2}\right)^{2m} m(m-1)\cdots((m-n)^++1) \\
	&\le C_nK \frac{d^{n}}{dx^{n}}\bigg|_{x=\pi/2} \cosh(Kx)
	\le  C_nK^{n+1}e^{K\pi/2}
\end{align*}
	(where the constants $C_n$ can and do vary between lines).
\end{proof}

\begin{lemma}
\label{l.1}
	Fix $K\ge1$. For all $n\in\mathbb{N}$ and $x\in(1,\infty)$,
\begin{equation} 
\label{e.l1} 
	\begin{split}
	\left|\frac{d^n}{dx^n} \cos(K\arccosh x)\right|
		&\le \lim_{x\downarrow1}\left|\frac{d^n}{dx^n} \cos(K\arccosh x)\right| \\
		&= \frac{(-1)^n}{(2n-1)!!} \prod_{m=0}^{n-1} (K^2+m^2). 
	\end{split}
	\end{equation}
\end{lemma}

\begin{proof}
	Let $G(K,x):=\cos(K \arccosh x)$. One may verify the form of the limit by induction as was done for (\ref{e.acos.lim}). The $n=1$ case is easily verifiable by direct computation:
	$G'(K,x)=-\frac{K\sin(K\arccosh x)}{\sqrt{x^2-1}}$ and $\lim_{x\downarrow 1} |G'(K,x)| = K^2$.
	More generally, and completely analogously to Lemma \ref{l.newnew1} and $\cosh(K\arccos x)$, one may show that for $n\ge2$ the derivatives are of the form
\begin{align*}
	G^{(n)}(K,x) &= \frac{N_n(K,x)}{(x^2-1)^{(2n-1)/2}} \\ 
	&= \frac{g_{n,c}(K,x)\sqrt{x^2-1}\cos(K \arccosh x) +g_{n,s}(K,x)\sin(K \arccosh x)}{(x^2-1)^{(2n-1)/2}}
\end{align*}
	where $g_{n,c}$ and $g_{n,s}$ are polynomials in $K$ and $x$ of degree $n-2$ and $n-1$ in $x$ respectively, so that $N_n(K,1)=0$. Again similarly to the proof of (\ref{e.mn}), one may show via induction that $N_n'=-(K^2+(n-1)^2)N_{n-1}$,
	and then, similarly to proof of (\ref{e.acos.lim}), induction combined with an application of L'H\^{o}pital's gives the desired limit.

	To show the inequality in \eqref{e.l1}, we work again by induction. One may verify the $n=1$ case directly that $|G'(K,x)|\le \lim_{x\downarrow 1} |G'(K,x)| = K^2$.
To deal with $n\ge2$, we note that the function $G^{(n)}$ of course has its critical points when $G^{(n+1)}=0$ and thus when $N_{n+1}=0$. Similarly to the proof of Lemma \ref{l.newnew1}, we have that
		\begin{equation*}
		N_{n+1}= -(K^2+(n-1)^2)N_{n-1}(x^2-1) - N_n(2n-1)x. 
	\end{equation*}
	Thus for any critical point $x_0>1$ of $G_n$
	\[ N_n(K,x_0) = -\frac{(K^2+(n-1)^2)N_{n-1}(K,x_0)(x_0^2-1)}{(2n-1)x_0}, \]
and so
\begin{align*}
	|G^{(n)}(K,x_0)| = \frac{|N_n(K,x_0)|}{(x_0^2-1)^{(2n-1)/2}} 
		&= \frac{(K^2+(n-1)^2)}{(2n-1)x_0}
			\frac{|N_{n-1}(K,x_0)|}{(x_0^2-1)^{(2n-3)/2}} \\
		&\le \frac{(K^2+(n-1)^2)}{(2n-1)} |G^{(n-1)}(K,x_0)|.
\end{align*}
	By the induction hypothesis, $G^{(n-1)}$ is dominated by (the absolute value of) its limit as $x\downarrow1$. This shows that all local maxima and minima of $G^{(n)}$ are dominated by its own limit as $x\downarrow1$ and thus the estimate holds for all $x>1$.
\end{proof}

\begin{lemma}
\label{l.2}
Fix $K\ge1$. For all $n\in\mathbb{N}$,
\[ \lim_{x\downarrow1}\left|\frac{d^n}{dx^n} \frac{\sin(K \arccosh x)}{\arccosh x}\right|
	\le C_n K^{2n+1}. \]
Also, for all $n\in\mathbb{N}$ and $x\in\left(1,\cosh\left(\frac{\pi}{2}\right)\right)$,
\[ \left|\frac{d^n}{dx^n} \frac{\sin(K \arccosh x)}{\arccosh x}\right|
	\le C_n K^{2n+1}e^{K\pi/2}. \]
\end{lemma}

\begin{proof}
Similar to the proof of  Lemma \ref{l.2-new}, we may show that
\begin{equation}
\label{e.sinK}
\lim_{x\downarrow1}\frac{d^{n}}{dx^{n}} \frac{\sin(K \arccosh x)}{\arccosh x}
	= \sum_{\ell=0}^\infty \frac{(-1)^\ell}{(2\ell+1)!}K^{2\ell+1} 
		\lim_{x\downarrow1}\frac{d^{n}}{dx^{n}}(\arccosh x)^{2\ell}.
\end{equation}
and
\[ \lim_{x\downarrow1}\frac{d^{n}}{dx^{n}}\cos(K\arccosh x)
	= \sum_{m=0}^\infty \frac{(-1)^\ell}{(2\ell)!}K^{2\ell}
		\lim_{x\downarrow1}\frac{d^{n}}{dx^{n}}(\arccosh x)^{2\ell}. \]
By Lemma \ref{l.1} 
\[ \lim_{x\downarrow1}\frac{d^{n}}{dx^{n}}\cos(K\arccosh x)
	= \frac{(-1)^n}{(2n-1)!!}\sum_{\ell=1}^n a_{\ell,n} K^{2\ell}
\]
	for some coefficients $a_{\ell,n}>0$. Thus comparing coefficients gives, for all $\ell>n$,
\[ \lim_{x\downarrow1}\frac{d^{n}}{dx^{n}}(\arccosh x)^{2\ell} =0, \]
and, for $\ell=1,\ldots, n$,
\[ \lim_{x\downarrow1}\frac{d^{n}}{dx^{n}}(\arccosh x)^{2\ell} 
	= (-1)^{n-\ell}a_{\ell,n} \frac{(2\ell)!}{(2n-1)!!}. \]
Thus, combining these limits with (\ref{e.sinK}) gives
\begin{align*} 
\lim_{x\downarrow1}\frac{d^{n}}{dx^{n}} \frac{\sin(K \arccosh x)}{\arccosh x}
	&= \frac{(-1)^n}{(2n-1)!!}\sum_{\ell=1}^n \frac{a_{\ell,n}}{2\ell+1}K^{2\ell+1},
\end{align*}
giving the desired bound on the limit.

Finally, the above estimates also imply that
	\begin{align*}
\left|\frac{d^{n}}{dx^{n}} \frac{\sin(K \arccosh x)}{\arccosh x}\right|
	&= \left|\sum_{\ell=0}^\infty \frac{(-1)^\ell}{(2\ell+1)!}K^{2\ell+1} 
		\frac{d^{n}}{dx^{n}}(\arccosh x)^{2\ell}\right| \\
	&\le K\sum_{\ell=0}^\infty \frac{K^{2\ell}}{(2\ell)!} 
		\left|\frac{d^{n}}{dx^{n}}(\arccosh x)^{2\ell}\right|.
\end{align*}
	Thus, writing the expression analogous to (\ref{e.acos2m}) for $\arccosh x$, and using the estimates from Lemma \ref{l.newnew1}, the proof concludes as in Lemma \ref{l.2-new}.
\end{proof}

\begin{remark}
	We will not directly use the bound on the limit of $\frac{d^n}{dx^n} \frac{\sin(K \arccosh x)}{\arccosh x}$. However, as in Lemma \ref{l.1}, it should be true that $\frac{d^{n}}{dx^{n}} \frac{\sin(K \arccosh x)}{\arccosh x}$ is dominated by its limit as $x$ approaches 1, giving a bound of 
	\[ \left|\frac{d^{n}}{dx^{n}} \frac{\sin(K \arccosh x)}{\arccosh x}\right|\le C_n K^{2n+1} \]
for all $x>1$. However, we do not give that proof here and the second bound in the statement of Lemma \ref{l.2} suffices for our purpose.
\end{remark}

\begin{lemma} 
\label{l.ect}
	For all $c>0$, $n\ge0$, and $t\in(0,1)$,
\[ \sum_{k=1}^\infty e^{-ck^2/t}k^n\le \sum_{k=1}^\infty e^{-ck/t}k^n \le C_n e^{-c/t}. \]
\end{lemma}

\begin{proof} The first inequality is trivial. For the second, we have that
\begin{align*}
\sum_{k=1}^\infty e^{-ck/t}k^n 
	&= \frac{e^{-c/t}}{1-e^{-c/t}}
		\sum_{k=1}^\infty \left(e^{-c/t}\right)^{k-1}\left(1-e^{-c/t}\right) k^n.
\end{align*}
	Note that the sum is the $n^{\mathrm{th}}$ moment of a geometric random variable with $t$ dependent parameter $p(t)=1-e^{-c/t}$, and is thus on the order of $1/p(t)$ (with implicit constant depending on $n$). Of course, $(1-e^{-c/t})^{-1}$ is bounded on $(0,1)$, and this completes the estimate.
\end{proof}

\subsection{Estimates for $R_1$, $R_2$, $Q_1$, and $Q_2$}

In this section, we give uniform bounds in $x$ for the factors $Q_1$ and $Q_2$ and remainder terms $R_1$ and $R_2$, as well as their derivatives. 

\begin{notation}
For $n\ge0$, let 
\[ R_1^{(n)}(t,x) := \frac{\partial^n}{\partial x^n} R_1(t,x) \]
and similarly for $R_2$, $Q_1$, and $Q_2$.
\end{notation}

\begin{lemma}
\label{l.R1}
	For all $n\ge0$, for all $t\in(0,1)$ and $x\in(0,1)$
\begin{align*}
\left|R_1^{(n)}(t,x)\right|
	\le C_n \frac{e^{-\pi^2/t}}{t^{n+1}}.
\end{align*}
\end{lemma}

\begin{proof}
	By Lemmas \ref{l.1-new} and \ref{l.2-new},
\begin{align*} 
\left|\frac{\partial^{n}}{\partial x^{n}} \cosh\left( \frac{\pi k \arccos x}{t}\right) \right|
	&\le C_n \left(\frac{k}{t}\right)^ne^{\pi^2 k/2t}\end{align*}
and 
\begin{align*}
&\left|\frac{\partial^n}{\partial x^n}\left(\frac{2\pi k}{\arccos x}\sinh\left(\frac{\pi k \arccos x}{t}\right) \right) \right|
	\le C_{n} \frac{ k^{n+2}}{t^{n+1}} e^{\pi^2 k/2t} .
\end{align*}
Since for $k\ge1$
	\begin{equation}
	\label{e.k}e^{-\pi^2k^2/t}  e^{\pi^2k/2t} k^{n+1} \le e^{-\pi^2k/2t} k^{n+1} \end{equation}
which is summable in $k$,  
	\begin{multline*} 
R_1^{(n)}(t,x) =
	 2\sum_{k=1}^\infty e^{-\pi^2k^2/t}\bigg(\frac{\partial^n}{\partial x^n}\cosh\left(\frac{\pi k \arccos x}{t}\right) \\
	+\frac{\partial^n}{\partial x^n}\frac{2\pi k}{\arccos x}\sinh\left(\frac{\pi k \arccos x}{t}\right)\bigg)
\end{multline*}
for all $x\in(0,1)$. Thus,
\begin{align*}
|R_1^{(n)}(t,x)|
	\le \frac{C_n}{t^{n+1}}\sum_{k=1}^\infty e^{-\pi^2k/2t} k^{n+2} ,
\end{align*}
and the result follows from Lemma \ref{l.ect}.
\end{proof}

\begin{lemma}
\label{l.R2}
	For all $n\ge0$, $t\in(0,1)$, and $x\in\left(1,\cosh\left(\frac{\pi}{2}\right)\right)$,
\[ \left|R_2^{(n)}(t,x)\right|
	\le C_{n} \frac{e^{-\pi^2/t}}{t^{2n+1}}. \]
and for all $x\in\left[\cosh\left(\frac{\pi}{2}\right),\infty\right)$
\[ \left|R_2^{(n)}(t,x)\right|
	\le \frac{C_{n}}{(x^2-1)^{n/2}} \frac{e^{-\pi^2/t}}{t^{n}}. \]
\end{lemma}

\begin{proof}
Lemma \ref{l.1} implies that 
\begin{align*}
\left|\frac{\partial ^n}{\partial x^n} \cos\left(\frac{\pi k \arccosh x}{t}\right)\right|
	&\le \lim_{x\downarrow1} \left|\frac{\partial ^n}{\partial x^n} \cos\left( \frac{\pi k \arccosh x}{t}\right)
		\right|\\
	&= \frac{1}{(2n-1)!!} \prod_{m=0}^{n-1} \left(\left(\frac{\pi k}{t}\right)^2+m^2\right)
	\le C_n \left(\frac{k}{t}\right)^{2n}
\end{align*}
	for all $x\in(1,\infty)$.
	Similarly, using Lemma \ref{l.2} and \eqref{e.k} one may show
\[ \left|\frac{\partial^n}{\partial x^n} \frac{2\pi k}{\arccosh x}
		\sin\left( \frac{\pi k \arccosh x}{t}\right)\right|
	\le C_n \frac{k^{2n+2}}{t^{2n+1}}e^{-\pi^2 k/2t} \] 
for all $x\in\left(1,\cosh\left(\frac{\pi}{2}\right)\right)$. Again by Lemma \ref{l.ect}, these bounds justify the interchange of sum and differentiation and give the desired estimate on $\left(1,\cosh\left(\frac{\pi}{2}\right)\right)$.

	Now for $x\in\left[\cosh\left(\frac{\pi}{2}\right),\infty\right)$, note first in general that
\begin{multline*}
\frac{\partial^n}{\partial x^n}\cos\left(\frac{\pi k \arccosh x}{t}\right)  \\
	=  \sum_{\alpha\in\mathcal{J}_n} c_{\alpha n}
		\left(\frac{d^{|\alpha|}}{du^{|\alpha|}}\cos u\right)
			\bigg|_{u=\frac{\pi k\arccosh x}{t}}
		\left(\frac{\pi k}{t}\right)^{|\alpha|} \prod_{j=1}^n
		\left(\frac{d^j}{dx^j}\arccosh x\right)^{\alpha_j}.
\end{multline*} 
Clearly,
\[ \left|\frac{d^{|\alpha|}}{du^{|\alpha|}}\cos u \bigg|_{u=\frac{\pi k\arccosh x}{t}}
	\right|\le 1. \]
We also have that 
\begin{align*}
\frac{d^{j}}{dx^{j}} & \arccosh x \\
	&= \sum_{k=0}^{\lfloor j/2\rfloor} 
		\frac{(-1)^{j-k-1}(j-1)!}{k!(j-2k-1)!}\frac{(2(j-k-1)-1)!!}{2^k}
		\frac{x^{j-2k-1}}{(x^2-1)^{2(j-k-1)+1)/2}} \\
	&= \sum_{k=0}^{\lfloor j/2\rfloor} 
		\frac{(-1)^{j-k-1}(j-1)!}{k!(j-2k-1)!}\frac{(2(j-k)-3)!!}{2^k}
		\left(\frac{x}{\sqrt{x^2-1}}\right)^{j-2k-1}
		\frac{1}{(x^2-1)^{j/2}},
\end{align*}
Thus, letting $a_{jk}$ denote the coefficients in the above expression,
\begin{align*} 
\prod_{j=1}^n \left(\frac{d^j}{dx^j}\arccosh x\right)^{\alpha_j}
	&= \prod_{j=1}^n \left(\sum_{k=0}^{\lfloor j/2\rfloor} a_{jk} \left(\frac{x}{\sqrt{x^2-1}}\right)^{j-2k-1}
		\frac{1}{(x^2-1)^{j/2}}\right)^{\alpha_j} \\
	&= \frac{1}{(x^2-1)^{n/2}}\prod_{j=1}^n \left(\sum_{k=0}^{\lfloor j/2\rfloor} a_{jk} \left(\frac{x}{\sqrt{x^2-1}}\right)^{j-2k-1}
		\right)^{\alpha_j}.
\end{align*}
Thus, 
\begin{equation*}
\frac{\partial^n}{\partial x^n}\left( \cos\left(\frac{\pi k \arccosh x}{t}\right)\right)
	\le \frac{C_{n}}{(x^2-1)^{n/2}} \left(\frac{ k}{t}\right)^n,
\end{equation*}
for all $x\in\left[\cosh\left(\frac{\pi}{2}\right),\infty\right)$ and $t\in(0,1)$.
Similarly, using that $1/\arccosh x$ is bounded when $x$ is bounded away from 1, we may show that 
\begin{align*}
\left|\frac{\partial^n}{\partial x^n}\left(\frac{2\pi k}{\arccosh x}\sin\left(\frac{\pi k \arccosh x}{t}\right) \right) \right|
	\le \frac{C_{n}}{(x^2-1)^{n/2}}\frac{ k^{n+1}}{t^n},
\end{align*}
and again Lemma \ref{l.ect} gives the desired estimate.
\end{proof}

\begin{lemma}
\label{l.Q1}
For all $n\ge0$, $x\in[0,1)$, and $t\in(0,1)$,
\[ \left|Q_1^{(n)}(t,x)\right| \le \frac{C_n}{t^n} . \]
\end{lemma}

\begin{proof}
First note that
\[ \frac{\partial}{\partial x} e^{-\arccos^2 x/4t} = e^{-\arccos^2 x/4t}\cdot \frac{1}{2t} \frac{\arccos x}{\sqrt{1-x^2}} \]
Thus,
\[  Q_1^{(n)}(t,x) 
	= \frac{\partial^n}{\partial x^n} \frac{\arccos x}{\sqrt{1-x^2}} e^{-\arccos^2 x/4t}  
	= 2t \cdot \frac{\partial^{n+1}}{\partial x^{n+1}} e^{-\arccos^2 x/4t}. \]
Now again using the notation of (\ref{chain}), we have that
\begin{align*}
\frac{\partial^{n+1}}{\partial x^{n+1}} e^{-\arccos^2 x/4t}
	= \sum_{\alpha\in\mathcal{J}_{n+1}} c_{\alpha, n+1} 
		e^{-\arccos^2 x/4t}\left(-\frac{1}{4t}\right)^{|\alpha|}
		\prod_{j=1}^{n+1}
		\left(\frac{d^j}{dx^j}\arccos^2 x\right)^{\alpha_j}
\end{align*}
which shows that
\begin{align*}
\left|\frac{\partial^{n+1}}{\partial x^{n+1}} e^{-\arccos^2 x/4t}\right| 
	\le \frac{C_n}{t^{n+1}} \sum_{\alpha\in\mathcal{J}_{n+1}}
		\prod_{j=1}^{n+1}
		\left|\frac{d^j}{dx^j}\arccos^2 x\right|^{\alpha_j}.
\end{align*}
	Lemma \ref{l.newnew1} completes the proof.
\end{proof}

\begin{lemma}
\label{l.Q2}
For all $n\ge0$, there exists $C_n<\infty$ so that for all $x\in(1,\infty)$ and $t\in(0,1)$
\begin{equation*}
\left|Q_2^{(n)}(t,x)\right|  
	\le \frac{C_n}{t^n}e^{\arccosh^2 x/4t}\frac{(\arccosh x)^{n+1}}{(x^2-1)^{(n+1)/2}}.
\end{equation*}
\end{lemma}

\begin{proof}
Similarly to the proof of Lemma \ref{l.Q1}, we have that 
\[ Q_2^{(n)}(t,x) = \frac{\partial ^n}{\partial x^n} \frac{\arccosh x}{\sqrt{x^2-1}} e^{\arccosh^2 x/4t}  
	= 2t \cdot \frac{\partial ^{n+1}}{\partial x^{n+1}} e^{\arccosh^2 x/4t}. \]
	Recalling (\ref{e.jarccosh2}), we have that	
\begin{align}
\notag
\frac{\partial ^{n+1}}{\partial x^{n+1}}& e^{\arccosh^2 x/4t} 
	= \sum_{\alpha\in\mathcal{J}_{n+1}} c_{\alpha,n+1} 
		e^{\arccosh^2 x/4t} \frac{1}{(4t)^{|\alpha|}}
		\prod_{j=1}^{n+1}
		\left(\frac{d^j}{dx^j}\arccosh^2 x\right)^{\alpha_j} \\ 
	&\label{e.nearccosh2}= \frac{e^{\arccosh^2 x/4t}}{(x^2-1)^{(n+1)/2}}\sum_{\alpha\in\mathcal{J}_{n+1}}  
		\frac{c_{\alpha ,n+1}}{(4t)^{|\alpha|}}
		\prod_{j=1}^{n+1}
		\left(\frac{p_j(x)\sqrt{x^2-1}+q_j(x)\arccosh x}{(x^2-1)^{(j-1)/2}}\right)^{\alpha_j}.
\end{align}
	where $p_1\equiv1$ and otherwise $p_j$ and $q_j$ are polynomials of degree $j-2$ and $j-1$. In particular, each factor in the product is bounded by $C(\arccosh x)^{\alpha_j}$ for large $x$, gives the desired bound.
\end{proof}

\section{Asymptotics of Hermite functions on $SU(2)$}
\label{s.hermite}

We now consider the behavior of the left-invariant vector fields for the Pauli basis under the given spatial scaling.
The vector fields in the cylindrical coordinates are given by
\begin{align*} 
\tilde{X} &= \cos(2z -\theta) \frac{\partial}{\partial r} + \tan r\sin(2z-\theta) \frac{\partial}{\partial z} +  \frac{2}{\sin 2r}\sin(2z-\theta) \frac{\partial}{\partial \theta},\\
\tilde{Y} &= -\sin(2z-\theta) \frac{\partial}{\partial r} + \tan r \cos(2z-\theta) \frac{\partial}{\partial z} + \frac{2}{\sin 2r} \cos(2z-\theta)  \frac{\partial}{\partial \theta},\\
\tilde{Z} &= \frac{\partial}{\partial z}.
\end{align*}
We also will need the left-invariant vector fields on the Heisenberg group $\mathbb{H}$ for the given basis expressed in cylindric coordinates:
\begin{align*}
\tilde{\mathcal{X}} &= \cos\theta \frac{\partial}{\partial r}  - r \sin\theta \frac{\partial}{\partial z} - 
	\frac{\sin\theta}{r}\frac{\partial}{\partial \theta}\\
\tilde{\mathcal{Y}} &= \sin\theta \frac{\partial}{\partial r}  + r \cos\theta \frac{\partial}{\partial z} + \frac{\cos\theta}{r}\frac{\partial}{\partial \theta}\\
\tilde{\mathcal{Z}} &= \frac{\partial}{\partial z}.
\end{align*}

Now, for any smooth function $f:SU(2)\rightarrow \mathbb{R}$, consider
\begin{align*}
\sqrt{t}(\tilde{X}f)&(\sqrt{t}r,\theta,tz) \\
	&= \sqrt{t}\bigg( (\cos(2z-\theta)(\partial_r f))(\sqrt{t}r,\theta,tz) 
		+ (\sin(2z-\theta)\tan r(\partial_z f))(\sqrt{t}r,\theta,tz) \\
	&\qquad\qquad +  \left(\frac{2}{\sin (2r)}\sin(2z-\theta) 
			(\partial_\theta f)\right)(\sqrt{t}r,\theta,tz)\bigg) \\
	&= \cos(2tz-\theta)\partial_r (f(\sqrt{t}r,\theta,tz)) 
		+ \frac{1}{\sqrt{t}}\sin(2tz-\theta)\tan (\sqrt{t}r)\partial_z
		(f(\sqrt{t}r,\theta,tz)) \\
	&\qquad\qquad +  \frac{2\sqrt{t}}{\sin (2\sqrt{t}r)}\sin(2tz-\theta) 
			\partial_\theta (f(\sqrt{t}r,\theta,tz)).
\end{align*}
Thus, if we define the vector field
\[ \tilde{X}^t := \cos(2tz-\theta)\partial_r 
		+ \sin(2tz-\theta) \frac{\tan (\sqrt{t}r)}{\sqrt{t}}\partial_z
		+  \frac{2\sqrt{t}}{\sin (2\sqrt{t}r)}\sin(2tz-\theta) 
			\partial_\theta, \]
then we have that
\begin{align*} 
\sqrt{t}(\tilde{X}f)(\sqrt{t}r,\theta, tz)
	= \tilde{X}^t (f(\sqrt{t}r,\theta, tz))).
\end{align*}
Note further that
\[ \lim_{t\downarrow 0} \tilde{X}^t = \tilde{\mathcal{X}} \]
the left-invariant vector field on the Heisenberg group. Similarly we may define
\begin{align*} 
	\tilde{Y}^t &:= -\sin(2tz-\theta) \frac{\partial}{\partial r} 
	+  \cos(2tz-\theta)\frac{\tan (\sqrt{t}r)}{\sqrt{t}} \frac{\partial}{\partial z} 
	+ \frac{2\sqrt{t}}{\sin (2\sqrt{t}r)} \cos(2tz-\theta)  \frac{\partial}{\partial \theta} \\
	\tilde{Z}^t &:= \frac{\partial}{\partial z},
\end{align*}
and we see that $\lim_{t\downarrow 0} \tilde{Y}^t = \tilde{\mathcal{Y}}$ and $\lim_{t\downarrow 0}\tilde{Z}^t=\tilde{\mathcal{Z}}$.
For $\xi=\xi_1\otimes\cdots\otimes\xi_n\in\mathfrak{su}(2)$ such that each $\xi_i\in\{X,Y,Z\}$, we define the differential operators $\tilde{\xi}^t$ analogously (for example, for $\xi=X\otimes X$ identified with the second order differential operator $\tilde{X}^2$, we let $\tilde{\xi}^t=(\tilde{X}^t)^2$).
One may verify that for any such $\xi$,  
\begin{equation*}
	\lim_{t\downarrow 0} \tilde{\xi}^t = \widetilde{\beta(\xi_1)}\cdots\widetilde{\beta(\xi_n)} 
		= \widetilde{\beta(\xi)}
\end{equation*}
where $\beta$ is the mapping defined in Section \ref{s.statements} and so the right hand side is a left-invariant $n^\mathrm{th}$ order  differential operator on the Heisenberg group.
(It's easy to see that the derivatives in $z$ and $\theta$ of the coefficients of the vector fields behave correctly in the limit, and only slightly more difficult to verify that the derivatives in $r$ also give the correct limiting behavior.)

We will show that, for any $m,n\ge0$
\[ \lim_{t\downarrow0} 
	t^2 \frac{\partial^m}{\partial z^m} \frac{\partial^n}{\partial r^n}(p_t(\sqrt{t}r,tz))
	= 2\pi^2 \frac{\partial^m}{\partial z^m}\frac{\partial^n}{\partial r^n} h_1(r,z). \]
By the above discussion, this then suffices to
prove Theorem \ref{t.main}. To further simplify computations, note that the integrand of both the heat kernels of $SU(2)$ and $\mathbb{H}$ factor into functions of $r$ and $z$, and we have that
\begin{multline*} 
	\frac{\partial^m}{\partial z^m}\frac{\partial^n}{\partial r^n} p_t(\sqrt{t}r,tz) \\
	= \frac{1}{\sqrt{4\pi t}} \int_{-\infty}^\infty 
		\frac{\partial^m}{\partial z^m} (e^{-iz\lambda/2}e^{-tz^2/4}) e^{-\lambda^2/4t}\frac{\partial^n}{\partial r^n} q_t(\cos \sqrt{t}r \cosh \lambda) \,d\lambda, 
\end{multline*}
and 
\[ \frac{\partial^m}{\partial z^m}\frac{\partial^n}{\partial r^n} h_1(r,z) 
	= \frac{1}{16\pi^2} \int_{-\infty}^\infty 
		\frac{\partial^m}{\partial z^m}e^{i\lambda z/2}\frac{\lambda}{\sinh\lambda} 
		 \frac{\partial^n}{\partial r^n} e^{-r^2\lambda\coth\lambda/4} \,d\lambda. \]
We essentially show that, as $t\downarrow 0$,
\[ \frac{\partial^m}{\partial z^m} (e^{-iz\lambda/2}e^{-tz^2/4}) 
	\rightarrow \frac{\partial^m}{\partial z^m}e^{-iz\lambda/2}
	\]
and
\[ \frac{t^2}{\sqrt{4\pi t}}e^{-\lambda^2/4t}
	\frac{\partial^n}{\partial r^n} q_t(\cos \sqrt{t}r \cosh \lambda) 
		\rightarrow \frac{1}{8} \frac{\lambda}{\sinh\lambda} 
			\frac{\partial^n}{\partial r^n}e^{-r^2\lambda\coth\lambda/4} \]
(remembering the factor of $\sqrt{\pi}/4t^{3/2}$ in $q_t$). We will demonstrate that we have sufficient control over all terms to interchange the limits and integration.
From our estimates, one may also show that we have sufficient control over products of these differentiated factors (for example, by separating $e^{-\lambda^2/4t}$ into $e^{-\lambda^2/8t}\cdot e^{-\lambda^2/8t}$ and weighting both factors to get integrability), and thus
it suffices to separate the derivatives. In particular, we will show that for any $m\ge0$
\begin{equation}
\label{e.crit1} \lim_{t\downarrow0} 
	t^2 \frac{\partial^m}{\partial z^m} (p_t(\sqrt{t}r,tz))
	= 2\pi^2\frac{\partial^m}{\partial z^m} h_1(r,z), 
\end{equation}
and for any $n\ge0$
\begin{equation}
\label{e.crit2}
\lim_{t\downarrow0} 
	t^2 \frac{\partial^n}{\partial r^n} (p_t(\sqrt{t}r,tz))
	= 2\pi^2 \frac{\partial^n}{\partial r^n}h_1(r,z), 
\end{equation}
which we prove in Propositions \ref{p.dz} and \ref{p.dr}, respectively. 

\begin{remark}
	Note that the scaling for the factor of $t$ is $t^{Q/2}$ where $Q=4$ is the homogeneous dimension of $\mathbb{H}^3$, the nilpotentisation of $SU(2)$, and thus corresponds to the correct scaling of the $SU(2)$ heat kernel away from the cut locus.
\end{remark}

Next we give some straightforward lemmas to demonstrate we have control over all quantities to perform the necessary analysis. First, a simple estimate that will be used several times in the sequel.

\begin{lemma}
\label{l.fr}
For any $n\ge 0$ and smooth function $f:\mathbb{R}\rightarrow\mathbb{R}$,
\begin{multline*}
\left|\frac{\partial^n}{\partial r^n} f\left(\cos(\sqrt{t}r)\cosh\lambda\right)\right| \\
	\le C_nt^{n/2}\sum_{\alpha\in\mathcal{J}_n} \left|f^{(|\alpha|)}\left(\cos (\sqrt{t}r) \cosh \lambda  \right)\right|  (\cosh \lambda )^{|\alpha|} 		
		\left(\sin(\sqrt{t}r)\right)^{|\alpha|_{odd}}
\end{multline*}
for all $\lambda\in\mathbb{R}$, $r\ge0$, $t\in(0,1)$, where $|\alpha|_{odd}$ is as defined in Notation \ref{n.multi}.
\end{lemma}

\begin{proof} Using (\ref{chain}) we write
\begin{align} 
\notag\frac{\partial^n}{\partial r^n} &f\left(\cos(\sqrt{t}r)\cosh\lambda\right) \\
	&\notag= \sum_{\alpha\in\mathcal{J}_n} c_{\alpha n} f^{(|\alpha|)}\left(\cos (\sqrt{t}r) \cosh \lambda  \right) 
	\prod_{j=1}^{n} 
	\left(\left(t^{j/2}\frac{d^{j}}{d r^{j}}\cos\right)\left(\sqrt{t} r\right)\cosh \lambda \right)^{\alpha_j} \\
	&\label{e.AA2}= t^{n/2}\sum_{\alpha\in\mathcal{J}_n} c_{\alpha n} f^{(|\alpha|)}\left(\cos (\sqrt{t}r) \cosh \lambda  \right)  (\cosh \lambda )^{|\alpha|} \prod_{j=1}^{n} 
	\left(\left(\frac{d^{j}}{d r^{j}}\cos\right)\left(\sqrt{t} r\right) \right)^{\alpha_j}
\end{align}
from which the estimate follows.
\end{proof}

The next lemma gives rough bounds on the ``remainder'' terms $R_1$ and $R_2$, and in particular show that they and their derivatives are (uniformly) negligible for small $t$.

\begin{lemma}
\label{l.Rr}
Let $n\ge0$, and $r\ge0$ and $t\in(0,1)$ such that $\sqrt{t}r<\pi/4$. Then for all 
$|\lambda|<\arccosh\left(\frac{1}{\cos(\sqrt{t}r)}\right)$
\begin{equation}
\label{e.R1r} 
\left|\frac{\partial^n}{\partial r^n} R_1\left(t,\cos(\sqrt{t}r)\cosh\lambda\right)\right| 
	\le C_n\frac{e^{-\pi^2/t}}{t^{n/2+1}},
\end{equation}
and for all $|\lambda|>\arccosh\left(\frac{1}{\cos(\sqrt{t}r)}\right)$
\begin{equation}
\label{e.R2r}  \left|\frac{\partial^n}{\partial r^n} R_2\left(t,\cos(\sqrt{t}r)\cosh\lambda\right)\right|
	\le C_n\frac{e^{-\pi^2/t}}{t^{3n/2+1}} . 
\end{equation}
\end{lemma}

\begin{proof}
Applying Lemmas \ref{l.fr} and then \ref{l.R1} implies that
\begin{align*}
&\left|\frac{\partial^n}{\partial r^n} R_1\left(t,\cos(\sqrt{t}r)\cosh\lambda\right)\right| 
	\le C t^{n/2}\sum_{\alpha\in\mathcal{J}_n} \frac{e^{-\pi^2/t}}{t^{|\alpha|+1}} (\cosh \lambda )^{|\alpha|}
\end{align*}
	and (\ref{e.R1r}) follows from the restrictions on $\lambda$ and $\sqrt{t}r$.

	For $\cos(\sqrt{t}r)\cosh\lambda\in(1,\cosh(\pi/2)]$, the proof is exactly as for (\ref{e.R1r}) using instead the first bound from Lemma \ref{l.R2}. For $\cos(\sqrt{t}r)\cosh\lambda\in(\cosh(\pi/2),\infty)$, Lemma \ref{l.fr} and the second bound in Lemma \ref{l.R2} imply that
\begin{multline*}
\left|\frac{\partial^n}{\partial r^n} R_2\left(t,\cos(\sqrt{t}r)\cosh\lambda\right)\right| \\
	\le C t^{n/2}\sum_{\alpha\in\mathcal{J}_n} 
	\frac{1}{((\cos(\sqrt{t}r)\cosh\lambda)^2-1)^{|\alpha|/2}} 
	 \frac{e^{-\pi^2/t}}{t^{|\alpha|}}  	
		(\cosh\lambda)^{|\alpha|} \left(\sin(\sqrt{t}r)\right)^{|\alpha|_{odd}}.
\end{multline*}
This then gives (\ref{e.R2r}) since
\[ \frac{\cosh \lambda}{\sqrt{(\cos(\sqrt{t}r)\cosh\lambda)^2-1}} \]
is bounded for $\cos(\sqrt{t}r)\cosh\lambda$ bounded away from 1.
\end{proof}

\begin{prop}
\label{p.G}
For all $n\ge0$, there exists $G_n\in L^1(\mathbb{R})$ such that 
\[ \left|e^{-\lambda^2/4t}\frac{\partial^n}{\partial r^n}
	\left(q_t\left(\cos(\sqrt{t}r)\cosh\lambda\right)\right)\right|
	\le G_n(\lambda) r^n
\]
for all $\lambda\in\mathbb{R}$ and $r\ge0$ and $t\in(0,1)$ such that $\sqrt{t}r<\pi/4$.
\end{prop}

\begin{proof}
For $|\lambda|< \arccosh(1/\cos(\sqrt{t}r))$ (that is, $\cos(\sqrt{t}r)\cosh\lambda<1$), 
we must control $Q_1$ and $1+R_1$ and their derivatives up to order $n$, and similarly for $Q_2$ and $1+R_2$ and their derivatives up to order $n$ when
$|\lambda|> \arccosh(1/\cos(\sqrt{t}r))$.
The inequalities (\ref{e.R1r}) and (\ref{e.R2r}) imply that there exists $C<\infty$ such that for 
all $\ell=0,\ldots, n$ 
\begin{equation}
\label{e.A} 
\left|\frac{\partial^{\ell}}{\partial r^{\ell}}
		\left(1 + R_k\left(t,\cos(\sqrt{t}r)\cosh\lambda\right)\right) \right|
	\le C 
\end{equation}
for $k=1,2$ on their respective domains.

Lemmas \ref{l.fr} and \ref{l.Q1} imply that for $|\lambda|< \arccosh(1/\cos(\sqrt{t}r))$
\begin{multline}
\label{e.C}
\bigg|\frac{\partial^m}{\partial r^m} Q_1\left(t,\cos(\sqrt{t}r)\cosh\lambda\right)\bigg| \\
	\le C_m t^{m/2}\sum_{\alpha\in\mathcal{J}_m} \frac{1}{t^{|\alpha|}} (\cosh \lambda )^{|\alpha|} 	
		\left(\sin(\sqrt{t}r)\right)^{|\alpha|_{odd}} 
	\le Cr^m,
\end{multline}
since for each $\alpha\in\mathcal{J}_m$
\begin{equation}
\label{e.A1}
\begin{split}
\frac{t^{m/2}\left(\sin(\sqrt{t}r)\right)^{|\alpha|_{odd}}}{t^{|\alpha|}}
	&= \frac{t^{m/2}}{t^{|\alpha|_{even}+|\alpha|_{odd}/2}} 
		\left(\frac{\sin(\sqrt{t}r)}{\sqrt{t}}\right)^{|\alpha|_{odd}} \\
	&= \left(\frac{\sin(\sqrt{t}r)}{\sqrt{t}}\right)^{|\alpha|_{odd}}
	\le r^{|\alpha|_{odd}}
\end{split}
\end{equation}
by Remark \ref{r.m}. 

Similarly, Lemmas \ref{l.fr} and \ref{l.Q2} imply that for $|\lambda|> \arccosh(1/\cos(\sqrt{t}r))$
\begin{multline}
\label{e.D}
\left|\frac{\partial^n}{\partial r^n} Q_2\left(t,\cos(\sqrt{t}r)\cosh\lambda\right)\right| 
	\le C e^{\arccosh^2 (\cos (\sqrt{t}r) \cosh \lambda)/4t} \\
	\times \sum_{\alpha\in\mathcal{J}_n} 
		\Bigg\{\frac{t^{n/2}\left(\sin(\sqrt{t}r)\right)^{|\alpha|_{odd}}}{t^{|\alpha|}}
		\left(\frac{\cosh \lambda}{\sqrt{(\cos (\sqrt{t}r) \cosh \lambda)^2-1}}\right)^{|\alpha|} \\
	\times	\frac{\left(\arccosh(\cos (\sqrt{t}r) \cosh \lambda)\right)^{|\alpha|+1}}{\sqrt{(\cos \sqrt{t}r \cosh \lambda)^2-1}}\Bigg\} .
\end{multline}
Again, (\ref{e.A1}) ensures that all factors of $t$ have non-negative exponent. Note also that for $1<u\le \cosh (\lambda/2)$
\begin{multline}
\label{e.D1} 
e^{-\lambda^2/4t}e^{\arccosh^2u/4t}\frac{\left(\arccosh u\right)^{|\alpha|+1}}{\sqrt{u^2-1}} \\
	= e^{-\lambda^2/4t}e^{\arccosh^2u/4t}
		\frac{\arccosh u}{\sqrt{u^2-1}}\cdot\left(\arccosh u\right)^{|\alpha|} 
	\le e^{-\lambda^2/8} \left(\frac{\lambda}{2}\right)^{|\alpha|} 
\end{multline}
and for $\cosh (\lambda/2)\le u\le \cosh \lambda$
\begin{equation}
\label{e.D2} 
e^{-\lambda^2/4t}e^{\arccosh^2u/4t} \frac{\left(\arccosh u\right)^{|\alpha|+1}}{\sqrt{u^2-1}}
	\le \frac{ \lambda^{|\alpha|+1}}{\sinh(\lambda/2)}. 
\end{equation}
As the remaining factor is a bounded function for large $\lambda$, this completes the proof.
\end{proof}

\begin{lemma}
\label{l.drq}
For all $n\ge0$, $r\ge0$, and $\lambda\in\mathbb{R}$,
\[\lim_{t\downarrow 0} e^{-\lambda^2/4t}\frac{\partial^n}{\partial r^n} e^{\arccosh^2(\cos(\sqrt{t}r)\cosh\lambda)/4t}
	=  \frac{\partial^n}{\partial r^n} e^{-r^2\lambda\coth\lambda/4}. \]
\end{lemma}

\begin{proof} 
First we compute the right hand side using \eqref{chain}:
\begin{align}
\frac{\partial^{n}}{\partial r^{n}}e^{-r^2\lambda\coth\lambda/4}
	&\notag= \sum_{\alpha\in\mathcal{J}_n} c_{\alpha n} e^{-r^2\lambda\coth\lambda/4}
		\prod_{j=1}^n \left(- \frac{\partial^{j}}{\partial r^{j}}\frac{r^2\lambda\coth\lambda}{4}\right)^{\alpha_j} \\
	&\label{e.drh}=\sum_{\alpha_1+2\alpha_2=n} c_{\alpha n} e^{-r^2\lambda\coth\lambda/4}
		\left(- \frac{r\lambda\coth\lambda}{2}\right)^{\alpha_1}
		\left(- \frac{\lambda\coth\lambda}{2}\right)^{\alpha_2}.
\end{align}

Now for the left hand side, fix $r$ and $\lambda$ and assume $t$ is sufficiently small that $\cos(\sqrt{t}r)\cosh\lambda\ge 1$. Note that 
\begin{align*}
\frac{\partial^{n}}{\partial r^{n}}&e^{\arccosh^2(\cos(\sqrt{t}r)\cosh\lambda)/4t} \\
	&= \sum_{\alpha\in\mathcal{J}_n} c_{\alpha n}
		e^{\arccosh^2 (\cos(\sqrt{t}r)\cosh\lambda)/4t}
		\prod_{j=1}^{n}
		\left(\frac{\partial^{j}}{\partial r^{j}}
		\frac{\arccosh^2(\cos(\sqrt{t}r)\cosh\lambda)}{4t}\right)^{\alpha_j}.
\end{align*}
Then
\begin{equation}
\label{e.easy2}
		\lim_{t\downarrow0} e^{-\lambda^2/4t}e^{\arccosh^2(\cos(\sqrt{t}r)\cosh\lambda)/4t}
	= e^{-r^2\lambda\coth\lambda/4},
\end{equation}
and by (\ref{e.AA2}) for each $j$
\begin{multline*}
\frac{\partial^{j}}{\partial r^{j}}\arccosh^2(\cos(\sqrt{t}r)\cosh\lambda)
	= t^{j/2}\sum_{\beta\in\mathcal{J}_j} c_{\beta j} 
		\left(\frac{d^{|\beta|}}{dx^{|\beta|}} \arccosh^2 x\right) 
		\bigg|_{x=\cos(\sqrt{t}r)\cosh\lambda} \\
		\times (\cosh \lambda)^{|\beta|} (\cos(\sqrt{t}r))^{|\beta|_{even}}(\sin(\sqrt{t}r))^{|\beta|_{odd}}.
\end{multline*}
Thus,
\begin{multline*}
\prod_{j=1}^{n}\left(\frac{\partial^{j}}{\partial r^{j}}\frac{\arccosh^2(\cos(\sqrt{t}r)\cosh\lambda)}{4t}\right)^{\alpha_j}\\
	= \prod_{j=1}^{n}\left(\frac{t^{(j-2)/2}}{4}\sum_{\beta\in\mathcal{J}_j} c_{\beta j} 
		\Bigg(\frac{d^{|\beta|}}{dx^{|\beta|}} \arccosh^2 x\right) 
		\bigg|_{x=\cos(\sqrt{t}r)\cosh\lambda} \\
	\times
		(\cosh \lambda)^{|\beta|} (\cos(\sqrt{t}r))^{|\beta|_{even}}(\sin(\sqrt{t}r))^{|\beta|_{odd}}\Bigg)^{\alpha_j}
\end{multline*}
which will converge to 0 as $t\downarrow0$ for any $\alpha$ so that $\alpha_j>0$ for any $j\ge3$. That is, the only non-zero contributions in the limit come from the elements $\alpha=(\alpha_1,\alpha_2,0,\ldots,0)$ of $\mathcal{J}_n$ and thus
\begin{multline*}
\lim_{t\downarrow 0} e^{-\lambda^2/4t}\frac{\partial^n}{\partial r^n}e^{\arccosh^2(\cos(\sqrt{t}r)\cosh\lambda)/4t} 
	= \lim_{t\downarrow 0} e^{-\lambda^2/4t}e^{\arccosh^2(\cos(\sqrt{t}r)\cosh\lambda)/4t}  \\
	\times \sum_{\alpha_1+2\alpha_2=n} c_{\alpha n}	\left(-\frac{\sin(\sqrt{t}r)}{2\sqrt{t}}\frac{\arccosh(\cos(\sqrt{t}r)\cosh\lambda)}{\sqrt{(\cos(\sqrt{t}r) \cosh\lambda)^2-1}}\cosh\lambda \right)^{\alpha_1} \\
	\times
\left(-\frac{\cos(\sqrt{t}r)}{2}\frac{\arccosh(\cos(\sqrt{t}r)\cosh\lambda)}{\sqrt{(\cos(\sqrt{t}r) \cosh\lambda)^2-1}}\cosh\lambda + O\left(\sin^2(\sqrt{t}r)\right) \right)^{\alpha_2}.
\end{multline*}
We may easily see that
	\begin{equation}
		\label{e.easy1}
		\lim_{t\downarrow 0}\frac{\arccosh(\cos(\sqrt{t}r)\cosh\lambda)}
			{\sqrt{(\cos(\sqrt{t}r)\cosh\lambda)^2-1}}
	= \frac{\lambda}{\sinh\lambda},
\end{equation}
which combined with (\ref{e.drh}) completes the proof.
\end{proof}

\begin{prop}
\label{p.drq}
For all $n\ge0$, $r\ge0$, and $\lambda\in\mathbb{R}$,
\[\lim_{t\downarrow 0} e^{-\lambda^2/4t}\frac{\partial^n}{\partial r^n} Q_2\left(t,\cos(\sqrt{t}r)\cosh\lambda\right)
	= \frac{\lambda}{\sinh\lambda} 
		\frac{\partial^n}{\partial r^n} e^{-r^2\lambda\coth\lambda/4}. \]
\end{prop}

\begin{proof} 
Fix $r$ and $\lambda$ and assume $t$ is sufficiently small that $\cos(\sqrt{t}r)\cosh\lambda\ge 1$. Then we may write
\begin{multline*}
\frac{\partial^n}{\partial r^n}Q_2\left(t,\cos(\sqrt{t}r)\cosh\lambda\right) \\
	= \sum_{m=0}^n {n\choose m} \frac{\partial^{n-m}}{\partial r^{n-m}}
		\frac{\arccosh(\cos(\sqrt{t}r)\cosh\lambda)}
			{\sqrt{(\cos(\sqrt{t}r)\cosh\lambda)^2-1}}
		\frac{\partial^{m}}{\partial r^{m}}e^{\arccosh^2(\cos(\sqrt{t}r)\cosh\lambda)/4t}.
\end{multline*}
By Lemma \ref{l.fr} and equation (\ref{e.jarccosh2})
\begin{align*}
\frac{\partial^{n-m}}{\partial r^{n-m}} 
		&\frac{\arccosh(\cos(\sqrt{t}r)\cosh\lambda)}
			{\sqrt{(\cos(\sqrt{t}r)\cosh\lambda)^2-1}}\\
	&\qquad= \frac{1}{2\sqrt{t}\sin(\sqrt{t}r)\cosh\lambda} \cdot\frac{\partial^{n-m+1}}{\partial r^{n-m+1}}
		\arccosh^2(\cos(\sqrt{t}r)\cosh\lambda)\\
	&\qquad=  \sum_{\alpha\in\mathcal{J}_{n-m}} O\left(t^{(n-m)/2} \sin(\sqrt{t}r)^{|\alpha|_{\mathrm{odd}}}\right)
\end{align*}
and by Lemma \ref{l.fr} and equation (\ref{e.nearccosh2})
\begin{multline*}
\frac{\partial^m}{\partial r^m}
	e^{\arccosh^2(\cos(\sqrt{t}r)\cosh\lambda)/4t} \\
	= \sum_{\alpha\in\mathcal{J}_m} O\left(e^{\arccosh^2(\cos(\sqrt{t}r)\cosh\lambda)/4t} t^{m/2}\frac{\sin(\sqrt{t}r)^{|\alpha|_{\mathrm{odd}}}}{t^{|\alpha|}}\right)
\end{multline*}
as $t\downarrow0$. (Note that the implicit constants do depend on $r$ and $\lambda$, but we have control of these by the computations in Proposition \ref{p.G}). Recalling (\ref{e.easy2}) and that
again (\ref{e.A1}) guarantees that all exponents of $t$ appearing in this second bound are non-negative implies that the only non-zero contribution in the limit is the $n=m$ term. Thus,
\begin{multline*}
\lim_{t\downarrow 0} e^{-\lambda^2/4t}\frac{\partial^n}{\partial r^n}Q_2\left(t,\cos(\sqrt{t}r)\cosh\lambda\right) \\
	= \lim_{t\downarrow 0}\frac{\arccosh(\cos(\sqrt{t}r)\cosh\lambda)}
			{\sqrt{(\cos(\sqrt{t}r)\cosh\lambda)^2-1}}
		e^{-\lambda^2/4t}\frac{\partial^{n}}{\partial r^{n}}e^{\arccosh^2(\cos(\sqrt{t}r)\cosh\lambda)/4t},
\end{multline*}
and combining this with Lemma \ref{l.drq} and equation (\ref{e.easy1}) complete the proof.
\end{proof}

We now have all the necessary elements for the proofs of (\ref{e.crit1}) and (\ref{e.crit2}).

\begin{prop}
\label{p.dz}
For all $n\ge0$, uniformly on compact subsets of $[0,\infty)\times\mathbb{R}$,
\[\lim_{t\downarrow 0} t^2 \frac{\partial^n}{\partial z^n} (p_t(\sqrt{t}r,tz)) 
	= 2\pi^2 \frac{\partial^n}{\partial z^n} h_1(r,z).\]
\end{prop}

\begin{proof}
Let $K$ be a compact subset of $[0,\infty)\times\mathbb{R}$, and let $t>0$ be sufficiently small that $(\sqrt{t}r,tz)\in[0,\pi/4]\times[-\pi,\pi]$ for all $(r,z)\in K$.

First note that making the change of variables $\lambda\mapsto-\lambda$ we may write 
\[ p_t(\sqrt{t}r,tz)
	= \frac{1}{\sqrt{4\pi t}} e^{tz^2/4} \int_{-\infty}^\infty 
		e^{-\lambda^2/4t}e^{i\lambda z/2} q_t(\cos (\sqrt{t}r)\cosh \lambda) d\lambda. \]
For $m,n\ge1$, let $a_m$ and $b_n$ be defined by 
\[ \frac{\partial^m}{\partial z^m} e^{tz^2/4} =: e^{tz^2/4} a_m(t,z) \]
and
\begin{align*}
\frac{\partial^n}{\partial z^n}e^{i\lambda z/2}
	=: e^{i\lambda z/2}b_{n}(\lambda).
\end{align*}
Then the $a_m$'s are polynomials in $t$ and $z$, and the $b_n$'s are polynomials (of order $n$) in $\lambda$ and thus integrable against $e^{-\lambda^2/4t}$. With this notation (and taking $a_0,b_0\equiv1$) we may write
\begin{align*}
\frac{\partial^n}{\partial z^n}\left(p_t(\sqrt{t}r,tz)\right) 
	&= \frac{1}{\sqrt{4\pi t}} \sum_{m=0}^n \binom{n}{m} 
		e^{tz^2/4}a_m(t,z) \\
	&\qquad \times
		\int_{-\infty}^\infty e^{-\lambda^2/4t}
		\left(e^{i\lambda z/2}b_{n-m}(\lambda)\right)
		q_t(\cos (\sqrt{t}r)\cosh \lambda) d\lambda.
\end{align*}
Note that each term of the polynomials $a_m$ will be at least degree 1 in $t$, thus, for all $m\ge1$,
\[ \lim_{t\downarrow0} e^{tz^2/4}a_m(t,z) = 0. \]
So we have that 
\begin{align*}
\lim_{t\downarrow0} t^2 \frac{\partial^n}{\partial z^n}&\left(p_t(\sqrt{t}r,tz)\right) \\
	&= \lim_{t\downarrow0} \frac{t^{3/2}}{\sqrt{4\pi}}  
		e^{tz^2/4} \int_{-\infty}^\infty e^{-\lambda^2/4t}
		e^{i\lambda z/2}b_{n}(\lambda)
		q_t(\cos (\sqrt{t}r)\cosh \lambda) d\lambda \\
	&= \lim_{t\downarrow0} \frac{e^te^{tz^2/4}}{8}  
		 (J_1^n(t,r,z)+ J_2^n(t,r,z))   
\end{align*}
where
\begin{multline*} 
J_1^n(t,r,z) 
	:= \int_{\cos(\sqrt{t}r)\cosh\lambda<1} e^{-\lambda^2/4t}
		e^{i\lambda z/2}b_{n}(\lambda) \\
		\times Q_1\left(t,\cos (\sqrt{t}r)\cosh \lambda\right)\left(1+ R_1\left(t,\cos (\sqrt{t}r)\cosh \lambda\right)\right) d\lambda 
\end{multline*}
and
\begin{multline*}
J_2^n(t,r,z):= \int_{\cos(\sqrt{t}r)\cosh\lambda>1} e^{-\lambda^2/4t}
		e^{i\lambda z/2}b_{n}(\lambda) \\
		\times Q_2\left(t,\cos (\sqrt{t}r)\cosh \lambda\right)\left(1+ R_2\left(t,\cos (\sqrt{t}r)\cosh \lambda\right)\right) d\lambda. 
\end{multline*}

For $J_1^n$, note that for $\sqrt{t}r\in(0,\pi/4)$
\[ |\lambda| < \arccosh\left(\frac{1}{\cos(\sqrt{t}r)}\right)
	\le \arccosh(\sqrt{2}). \]
Combining this with (\ref{e.A}) with $\ell=0$ and (\ref{e.C}) with $m=0$ implies that the integrand of $J_1^n$ is uniformly bounded by a finite constant, and thus
\begin{align*} 
|J_1^n(t,r,z)| 
	&\le Cm\left(\left\{\lambda: 
		|\lambda|\le \arccosh\left(\frac{1}{\cos(\sqrt{t}r)}\right)\right\}\right) 
	\le C\sqrt{t}r\rightarrow 0
\end{align*}
as $t\downarrow0$, where $m(\{\cdot\})$ denotes Lebesgue measure of the set.

Now for $J_2^n$, Proposition \ref{p.G} implies that $Q_2(1+R_2)$ is uniformly bounded in $t$ by a function integrable in $\lambda$, thus
\begin{multline*} 
\lim_{t\downarrow0} J_2^n(t,r,z)
	= \int_{-\infty}^\infty \lim_{t\downarrow0} 
		1_{\cos(\sqrt{t}r)\cosh\lambda>1} e^{-\lambda^2/4t}
		e^{i\lambda z/2}b_{n}(\lambda) \\
		\times Q_2\left(t,\cos (\sqrt{t}r)\cosh \lambda\right)\left(1+ R_2\left(t,\cos (\sqrt{t}r)\cosh \lambda\right)\right) d\lambda. 
\end{multline*}
Lemma \ref{l.Rr} shows that $R_2\left(t,\cos(\sqrt{t}r)\cosh\lambda\right)\rightarrow0$ as $t\downarrow0$.

For derivatives of the Heisenberg kernel, note that
\begin{align*} 
\left|\frac{\partial^n}{\partial z^n}e^{i\lambda z/2}
		\frac{\lambda}{\sinh\lambda}e^{-r^2\lambda\coth\lambda/4}\right| 
	= \left|b_n(\lambda) \frac{\lambda}{\sinh\lambda} e^{-r^2\lambda\coth\lambda/4}\right| 
	\le C|\lambda|^{n+1} e^{-|\lambda|} 
\end{align*}
for all $r$, $z$, and sufficiently large $\lambda$. Thus,
\[\frac{\partial^n}{\partial z^n} h_1(r,z) 
	= \frac{1}{16\pi^2} \int_{-\infty}^\infty  e^{i \lambda z/2} b_n(\lambda)
		\frac{\lambda}{\sinh\lambda }e^{-r^2 \lambda \coth \lambda/4} d\lambda.\]
Thus, Proposition \ref{p.drq} with $n=0$ completes the proof.
\end{proof}

\begin{prop}
\label{p.dr}
For all $n\ge0$, uniformly on compact subsets of $[0,\infty)\times\mathbb{R}$,
\[\lim_{t\downarrow 0} t^2 \frac{\partial^n}{\partial r^n} (p_t(\sqrt{t}r,tz))
	= 2\pi^2 \frac{\partial^n}{\partial r^n} h_1(r,z).\]
\end{prop}

\begin{proof}
First note that, for derivatives of the Heisenberg kernel, equation (\ref{e.drh}) implies that
\[ \left|\frac{\lambda}{\sinh\lambda }\frac{\partial^n}{\partial r^n}e^{-r^2\lambda\coth\lambda/4}\right|
	\le C|\lambda|^{n+1}e^{-|\lambda|} \]
for all $r$ and sufficiently large $\lambda$. Since $|\lambda|^{n+1}e^{-|\lambda|}$ is integrable, this implies that
\[\frac{\partial^n}{\partial r^n} h_1(r,z) 
	= \frac{1}{16\pi^2} \int_{-\infty}^\infty  e^{i \lambda z/2} 
		\frac{\lambda}{\sinh\lambda }\frac{\partial^n}{\partial r^n} e^{-r^2 \lambda \coth \lambda/4} d\lambda.\]

Now let $K$ be a compact subset of $[0,\infty)\times\mathbb{R}$, and let $t>0$ be sufficiently small that $(\sqrt{t}r,tz)\in[0,\pi/4]\times[-\pi,\pi]$ for all $(r,z)\in K$.
Note that Proposition \ref{p.G} and equation (\ref{e.qtx}) allow us to write
\begin{align*}
t^2 \frac{\partial^n}{\partial r^n}( p_t(\sqrt{t}r,tz))
	&= \frac{t^{3/2}}{\sqrt{4\pi }} e^{tz^2/4} \int_{-\infty}^\infty 
		e^{-\lambda^2/4t}e^{i\lambda z/2} 
		\frac{\partial^n}{\partial r^n}q_t(\cos (\sqrt{t}r)\cosh \lambda) d\lambda \\
	&= \frac{e^te^{tz^2/4}}{8}  
		 (I^n_1(t,r,z)+ I^n_2(t,r,z))   
\end{align*}
where
\begin{multline*} 
I_1^n(t,r,z) 
	:= \int_{\cos(\sqrt{t}r)\cosh\lambda<1} e^{-\lambda^2/4t}
		e^{i\lambda z/2} \\
		\times \frac{\partial^n}{\partial r^n}\left[Q_1\left(t,\cos (\sqrt{t}r)\cosh \lambda\right)\left(1+ R_1\left(t,\cos (\sqrt{t}r)\cosh \lambda\right)\right)\right] d\lambda 
\end{multline*}
and
\begin{multline*}
I_2^n(t,r,z):= \int_{\cos(\sqrt{t}r)\cosh\lambda>1} e^{-\lambda^2/4t}
		e^{i\lambda z/2} \\
		\times \frac{\partial^n}{\partial r^n}\left[Q_2\left(t,\cos (\sqrt{t}r)\cosh \lambda\right)\left(1+ R_2\left(t,\cos (\sqrt{t}r)\cosh \lambda\right)\right)\right] d\lambda. 
\end{multline*}
As in the proof of Proposition \ref{p.dz} we may use (\ref{e.A}) and (\ref{e.C}) to say that $I_1^n (t,r,z)\rightarrow0$ as $t\downarrow0$.
	Also, for $I_2^n$, Proposition \ref{p.G} (and more specifically (\ref{e.A}), (\ref{e.D}), (\ref{e.D1}), and (\ref{e.D2})) show that $e^{-\lambda^2/4t}\frac{\partial^n}{\partial r^n}Q_2(1+R_2)$ is uniformly bounded in $t$ by a function integrable in $\lambda$, thus
\begin{multline*} 
\lim_{t\downarrow0} I_2^n(t,r,z)
	= \int_{-\infty}^\infty \lim_{t\downarrow0} 
		1_{\cos(\sqrt{t}r)\cosh\lambda>1} e^{-\lambda^2/4t}
		e^{i\lambda z/2} \\
		\times \frac{\partial^n}{\partial r^n}\left[Q_2\left(t,\cos (\sqrt{t}r)\cosh \lambda\right)\left(1+ R_2\left(t,\cos (\sqrt{t}r)\cosh \lambda\right)\right)\right] d\lambda. 
\end{multline*}
Lemma \ref{l.Rr} shows that $\frac{\partial^m}{\partial r^m}R_2\left(t,\cos(\sqrt{t}r)\cosh\lambda\right)\rightarrow0$ for all $m\ge0$ as $t\downarrow0$, 
and Proposition \ref{p.drq} completes the proof.
\end{proof}

\section{Asymptotics of Hermite functions on the CR sphere}

We now turn to the general case of odd-dimensional spheres $\mathbb{S}^{2d+1}$. We use here the notation and coordinates introduced in Section \ref{s.intro-sphere}. Additionally, for $w\in\mathbb{CP}^d$ let $\rho=\|w\|=\sqrt{\sum_{j=1}^d |w_j|^2}$ and define $r$ by $\rho=\tan r$. In Proposition 3.2 of \cite{BaudoinWang2013}, Baudoin and Wang prove that,
for $t>0$, $r\in[0,\pi/2)$, and $z\in[-\pi,\pi]$, the fundamental solution of the subelliptic heat equation on $\mathbb{S}^{2d+1}$ emitted from the north pole is given by
\[ p_{t,d}(r,z) = \frac{1}{\sqrt{4\pi t}}\int_\mathbb{R}e^{(\lambda+iz)^2/4t} q_{t,d}(\cos r\cosh\lambda)\,d\lambda \]
where $q_{t,d}(\cos\delta)$ is the fundamental solution to the Riemannian heat equation on $\mathbb{S}^{2d+1}$ (here $\delta$ is again the Riemannian distance from the north pole).

The Riemannian heat kernel for $\mathbb{S}^{2d+1}$ is well-known and may be written as
\[ q_{t,d}(\cos\delta) 
	= e^{d^2t}\left(-\frac{1}{2\pi\sin\delta} \frac{\partial}{\partial\delta}\right)^d V(t,\delta) \]
where
\begin{align*} 
V(t,\delta) &= \frac{1}{\sqrt{4\pi t}} \sum_{k\in\mathbb{Z}}e^{-(\delta+2k\pi)^2/4t},
\end{align*}
see for example Section 8 of \cite{Camporesi1990}. Note by the derivation of $q_{t,d+1}$ in that reference, or here by direct computation, that we have
\begin{align*} 
q_{t,d+1}(\cos\delta) 
	&= e^{(d+1)^2t}\left(-\frac{1}{2\pi\sin\delta} 
		\frac{\partial}{\partial\delta}\right)^{d+1} V(t,\delta) \\
	&= e^{(d+1)^2t}\left(-\frac{1}{2\pi\sin\delta} 
		\frac{\partial}{\partial\delta}\right)e^{-d^2t}e^{d^2t}\left(-\frac{1}{2\pi\sin\delta} 
		\frac{\partial}{\partial\delta}\right)^{d} V(t,\delta) \\
	&= e^{(d+1)^2t}e^{-d^2t}\left(-\frac{1}{2\pi\sin\delta} 
		\frac{\partial}{\partial\delta}\right)q_{t,d}(\cos\delta) \\
	&= e^{(2d+1)t}\cdot -\frac{1}{2\pi\sin\delta} q_{t,d}'(\cos\delta)\cdot-\sin\delta \\
	&= \frac{1}{2\pi}e^{(2d+1)t} q_{t,d}'(\cos\delta).
\end{align*}
For $\cos\delta=x$, we then have that for $d\ge2$
\begin{align*} 
q_{t,d}(x) &= \frac{1}{2\pi}e^{(2d-1)t} q_{t,d-1}'(x) \\
	&= \frac{1}{2\pi}e^{(2d-1)t} \cdot \frac{1}{2\pi}e^{(2d-3)t} q_{t,d-2}''(x) 
	= \left(\frac{1}{2\pi}\right)^{d-1} e^{(d^2-1)t} q_{t,1}^{(d-1)}(x).
\end{align*}
Note that $q_{t,1}$ is the same $q_t$ from Section \ref{s.qt}, modulo a factor of $\pi^2$ that comes from the choice of normalization of the Haar measure; that is, $\pi^2q_{t,1}=q_t$. Thus, using the notation adopted in Section \ref{s.qt}, and in particular (\ref{e.qtx}),
\begin{align} 
q_{t,d}(x) &\label{e.qs}= \frac{1}{\pi^2}\left(\frac{1}{2\pi}\right)^{d-1} e^{(d^2-1)t} q_{t}^{(d-1)}(x) \\
	&\notag= \left(\frac{1}{2\pi}\right)^{d+1}  \frac{\sqrt{\pi}e^{d^2t}}{t^{3/2}}
	\frac{\partial^{d-1}}{\partial x^{d-1}} 
	\left\{ \begin{array}{ll}
	Q_1(t,x) \left(1+ R_1(t,x) \right) & \text{for } x\in[0,1] \\
	Q_2(t,x) \left(1+ R_2(t,x) \right) & \text{for } x\ge1
	\end{array} \right. .
\end{align}

Now to perform the analysis, we must first re-write the vector fields in the inhomogeneous coordinates. Recall from Section \ref{s.intro-sphere} that for $k=1,\ldots,d$
\[ z_k = \frac{w_ke^{i z}}{\sqrt{1+\rho^2}} \quad \text{ and } \quad 
	z_{d+1} = \frac{e^{i z}}{\sqrt{1+\rho^2}}. \]
Thus, for $k=1,\ldots,d$,
\begin{align*} 
\frac{\partial}{\partial z_k} 
	&= \sqrt{1+\rho^2}e^{-i z}\frac{\partial}{\partial w_k} 
\end{align*}
and
\begin{align*} 
\frac{\partial}{\partial z_{d+1}} 
	&= -\sqrt{1+\rho^2}e^{-i z}\left(\sum_{k=1}^d w_k\frac{\partial}{\partial w_k} -\frac{1}{2i}\frac{\partial}{\partial z}\right)
\end{align*}
In particular, this gives that
\begin{multline*} 
\mathcal{S} = \sum_{k=1}^{d+1} z_k \frac{\partial}{\partial {z_k}} 
	= \sum_{k=1}^{d} \frac{w_ke^{i z}}{\sqrt{1+\rho^2}} \sqrt{1+\rho^2}e^{-i z}\frac{\partial}{\partial w_k} \\
	 + \frac{e^{i z}}{\sqrt{1+\rho^2}}\cdot -\sqrt{1+\rho^2}e^{-i z}\left(\sum_{k=1}^d w_k\frac{\partial}{\partial w_k} 
	-\frac{1}{2i}\frac{\partial}{\partial z}\right) 
	=  \frac{1}{2i}\frac{\partial}{\partial z}
\end{multline*}
and thus for $j=1,\ldots,d$
\begin{align*} 
T_j &= \frac{\partial}{\partial {z_j}} - \overline{z}_j\mathcal{S}
	= \sqrt{1+\rho^2}e^{-i z}\frac{\partial}{\partial w_j}
		- \frac{\overline{w}_je^{-i z}}{\sqrt{1+\rho^2}}
		\frac{1}{2i}\frac{\partial}{\partial z}.
\end{align*}
Recall that, for a smooth function $f:\mathbb{S}^{2d+1}\rightarrow\mathbb{R}$, we're interested in the following scaling for the vector fields
\begin{multline*}
\sqrt{t}(T_jf)(\sqrt{t}w_1,\ldots,\sqrt{t}w_n,t z)
	=\sqrt{1+t\rho^2}e^{-it z}\frac{\partial }{\partial w_j}\left(f(\sqrt{t}w_1,\ldots,\sqrt{t}w_n,t z)\right) \\
		- \frac{\overline{w}_je^{-it z}}{\sqrt{1+t\rho^2}}
		\frac{1}{2i}\frac{\partial }{\partial z}\left(f(\sqrt{t}w_1,\ldots,\sqrt{t}w_n,t z)\right).
\end{multline*}
Thus, if we take
\[ T_j^t:=\sqrt{1+t\rho^2}e^{-it z}\frac{\partial }{\partial w_j}
		- \frac{\overline{w}_je^{-it z}}{\sqrt{1+t\rho^2}}
		\frac{1}{2i}\frac{\partial }{\partial z} \] 
then
\[ \sqrt{t}(T_jf)(\sqrt{t}w_1,\ldots,\sqrt{t}w_n,t z) 
	= T_j^t( f(\sqrt{t}w_1,\ldots,\sqrt{t}w_n,t z)) \]
and
\begin{equation*}
T_j^t \rightarrow \frac{\partial }{\partial w_j}
		+ \frac{1}{2} i\overline{w}_j
		\frac{\partial }{\partial z} 
	= \mathcal{Z}_j(w_1,\ldots,w_d, z)
\end{equation*}
as $t\downarrow0$, where $\mathcal{Z}_j=\tilde{\mathcal{Y}}_j-i\tilde{\mathcal{X}}_j$ is the vector field on the $2d+1$-dimensional Heisenberg group $\mathbb{H}^{2d+1}$.

Similarly, we may write
\begin{align*}
T_{d+1} &= \frac{\partial}{\partial {z_{d+1}}} 
		- \overline{z}_{d+1}\sum_{k=1}^{d+1} z_k \frac{\partial}{\partial {z_k}} \\
	&= -\sqrt{1+\rho^2}e^{-i z}\left(\sum_{k=1}^d w_k\frac{\partial}{\partial w_k} -\frac{1}{2i}\frac{\partial}{\partial z}\right)
	- \frac{e^{-i z}}{\sqrt{1+\rho^2}} \frac{1}{2i}\frac{\partial}{\partial z} \\
	&= -\sqrt{1+\rho^2}e^{-i z}\sum_{k=1}^d w_k\frac{\partial}{\partial w_k} 
		+ \frac{e^{-i z}}{2i}\frac{\rho^2}{\sqrt{1+\rho^2}}\frac{\partial}{\partial z}.
\end{align*}
Note that
\begin{equation*}
T_{d+1} = -w\cdot T := -\sum_{k=1}^d w_kT_k. 
\end{equation*}
This implies that
\begin{align*}
(T_{d+1}f)(\sqrt{t}w_1,\ldots,\sqrt{t}w_d,t z) 
	&= -\left(\sum_{k=1}^d w_k T_kf\right)(\sqrt{t}w_1,\ldots,\sqrt{t}w_d,t z) \\
	&= -\sum_{k=1}^d \sqrt{t}w_k (T_kf) (\sqrt{t}w_1,\ldots,\sqrt{t}w_d,t z)\\
	&= -\sum_{k=1}^d w_k T_k^t (f(\sqrt{t}w_1,\ldots,\sqrt{t}w_d,t z)),
\end{align*}
and thus
\[ T_{d+1}^t := - w\cdot T^t
	:= - \sum_{k=1}^d w_k T_k^t
	\rightarrow 
	- \sum_{k=1}^d w_k\mathcal{Z}_k.
\] 
as $t\downarrow0$. Finally, note that, in the inhomogeneous coordinates
\[ T_0 = \frac{\partial}{\partial  z},\]
and thus
\[ t(T_0f)(\sqrt{t}w_1,\ldots,\sqrt{t}w_n,t z) = T_0(f(\sqrt{t}w_1,\ldots,\sqrt{t}w_n,t z)).\]

It is straightforward to verify that the derivatives of the coefficients of the $T_j$'s have the correct limiting values as the derivatives of the coefficients of the $\mathcal{Z}_j$'s on $\mathbb{H}^{2d+1}$. Recall that the subelliptic heat kernel is actually a function of the form $f=f(r, z)$, where
\[ r = \arctan \rho = \arctan\left(\sqrt{\sum_{j=1}^d w_j\overline{w}_j}\right), \]
In particular, for any such function
\begin{align*}
	\sqrt{t}\frac{\partial f}{\partial w_j} (\sqrt{t}r,tz) 
		= \frac{\partial}{\partial r} (f(\sqrt{t}r,tz)) \cdot \frac{\sqrt{t}\overline{w}_j}{2\sqrt{t}\rho(1+t\rho^2)}  
\end{align*}
where
\[ \frac{\sqrt{t}\overline{w}_j}{2\sqrt{t}\rho(1+t\rho^2)} = \frac{\overline{w}_j}{2\rho(1+t\rho^2)} \rightarrow
	\frac{\overline{w}_j}{2\rho}, \]
$t\downarrow0$. Note this is the correct limit for the Heisenberg object, since on $\mathbb{H}^{2d+1}$ we have $r=\sqrt{\sum_{j=1}^d w_j\overline{w}_j}$ and thus
\[ \frac{\partial f}{\partial w_j} (r,z)
		= \frac{\partial f}{\partial r} (r,z) \cdot \frac{\overline{w}_j}{2\rho}. \]

Again taking into account the product form of the integrand in the representation of $p_{t,d}$, analogously to the $SU(2)$ case, to prove Theorem \ref{t.sphere} it suffices to verify the following: For any $n\ge0$,
\begin{equation}
\label{e.dzn2} 
\lim_{t\downarrow0} 
	t^{d+1}\frac{\partial^n}{\partial z^n}  (p_{t,d}(\sqrt{t}r,tz) )
	= 2 \frac{\partial^n}{\partial z^n} h_{1,d}(r,z)
\end{equation}
and 
\begin{equation} 
\label{e.drn2}
\lim_{t\downarrow0} 
	t^{d+1} \frac{\partial^n}{\partial r^n} (p_{t,d}(\sqrt{t}r,tz))
	= 2 \frac{\partial^n}{\partial r^n}h_{1,d}(r,z)
\end{equation}
where 
\[ h_{t,d}(r,z) = \frac{1}{(4\pi)^{d+1}}\int_{\mathbb{R}} e^{i\lambda z/2}
	\left(\frac{\lambda}{\sinh \lambda t}\right)^d e^{-r^2\lambda\coth\lambda t/4}\,d\lambda.\]
is the subelliptic kernel on the Heisenberg group $\mathbb{H}^{2d+1}$ at time $t=1$; see for example \cite{Gaveau1977}.
Note again that this differs by a factor of $\pi^2$ from the $SU(2)$ limit due to the choice there of normalization of the volume measure. Also, the factor of $t$ appearing in (\ref{e.dzn2}) and (\ref{e.drn2}) is half the homogeneous dimension $Q=2d+2$ of the sub-Riemannian spheres. 

Given the relationship between the Riemannian kernel $q_{t,d}$ on $\mathbb{S}^{2d+1}$ and $\pi^2q_{t,1}=q_t$ on $\mathbb{S}^3\cong SU(2)$ identified in (\ref{e.qs}), much of the analysis of Section \ref{s.hermite} can be directly used to prove (\ref{e.dzn2}) and (\ref{e.drn2}). 
Thus, we provide here only a partial proof for the radial derivatives, including the main adaptation of the analysis used in the $SU(2)$ case. The next proposition is analogous to Proposition \ref{p.drq} and its proof is similar.

\begin{prop}
\label{p.drq2}
Let $n\ge0$, $r\ge0$, and $\lambda\in\mathbb{R}$. Then
\begin{multline*}
\lim_{t\downarrow 0} t^{d-1} e^{-\lambda^2/4t}\frac{\partial^n}{\partial r^n} Q_2^{(k)}\left(t,\cos(\sqrt{t}r)\cosh\lambda\right) \\
	=\left\{ \begin{array}{ll} 0 & \text{ if } k<d-1 \\
		2\left(\frac{\lambda}{2\sinh\lambda} \right)^d
		\frac{\partial^n}{\partial r^n} e^{-r^2\lambda\coth\lambda/4} &\text{ if } k=d-1
	\end{array} \right. . 
\end{multline*}
\end{prop}

\begin{proof} Fix $r$ and $\lambda$ and assume that $t$ is sufficiently small that $\cos(\sqrt{t}r)\cosh\lambda>1$. We have that
\begin{align*}
\frac{\partial^n}{\partial r^n} &Q_2^{(k)}\left(t,\cos(\sqrt{t}r)\cosh\lambda\right) \\
	&= t^{n/2}\sum_{\alpha\in\mathcal{J}_n}c_{\alpha n} 
		Q^{(k+|\alpha|)}(\cos\sqrt{t}r\cosh\lambda)(\cosh\lambda)^{|\alpha|} \prod_{\ell=1}^n\left(\left(\frac{\partial^\ell}{\partial r^\ell}\cos\right)(\sqrt{t} r)\right)^{\alpha_\ell}\\
	&= \sum_{\alpha\in\mathcal{J}_n} O\left(t^{n/2}\frac{1}{t^{k+|\alpha|}}\sin(\sqrt{t}r)^{|\alpha|_{\mathrm{odd}}}\right)
\end{align*}
as $t\downarrow0$. Since $k\le d-1$ and by Remark \ref{r.m}, it follows that
\[ d-1 + n/2 + |\alpha|_{\mathrm{odd}}/2 - (k+|\alpha|)\ge0 \]
where the inequality must be strict when $k<d-1$. Thus, the result follows for the case $k<d-1$.

For $k=d-1$, from the proof of Lemma \ref{l.Q2} we may write
\[
Q_2^{(d-1)}(t,x) 
	= 2te^{\arccosh^2 x/4t} \sum_{\alpha\in\mathcal{J}_{d}}  
		\frac{c_{\alpha d}}{(4t)^{|\alpha|}}
		F_\alpha(x) \]
where 
\[ F_\alpha(x) = \prod_{j=1}^{d}
		\left(\frac{d^j}{dx^j}\arccosh^2 x\right)^{\alpha_j}.\]
Thus, by Lemma \ref{l.fr}, (\ref{e.jarccosh2}), and (\ref{e.easy2})
\begin{align*}
&t^{d-1}e^{-\lambda^2/4t}\frac{\partial^n}{\partial r^n}Q_2^{(d-1)}(t,\cos(\sqrt{t}r)\cosh\lambda) \\
	&= t^{d-1}e^{-\lambda^2/4t} \\
	& \times 2t\sum_{j=0}^{n} {n\choose j} \frac{\partial^j}{\partial r^j}e^{\arccosh^2 (\cos(\sqrt{t}r)\cosh\lambda)/4t} \sum_{\alpha\in\mathcal{J}_{d}}  
		\frac{c_{\alpha d}}{(4t)^{|\alpha|}} \frac{\partial^{n-j}}{\partial r^{n-j}} F_\alpha(\cos(\sqrt{t}r)\cosh\lambda) \\
	&= \sum_{j=0}^{n} \sum_{\alpha\in\mathcal{J}_{d}} 
		\sum_{\beta\in\mathcal{J}_{j}}\sum_{\gamma\in\mathcal{J}_{n-j}} 
		O\left(	t^{d-1}\cdot\frac{1}{t^{|\beta|}}t^{j/2}\sin(\sqrt{t}r)^{|\beta|_{\mathrm{odd}}} 
		\cdot \frac{t^{(n-j)/2}}{t^{|\alpha|-1}}\sin(\sqrt{t}r)^{|\gamma|_{\mathrm{odd}}}.  \right)
\end{align*}
(Again the implicit constants do depend on $\lambda$ and $r$, but in a controllable way.) Note that 
	\[ \{d-1 + (j/2+|\beta|_{\mathrm{odd}}/2) + ((n-j)/2 + |\gamma|_{\mathrm{odd}}/2)\} - \{|\alpha|-1 + |\beta|\} \ge 0 \]
with equality only when $j= n$ and $\alpha=(d,0,\ldots0)$, where
\[ F_{(d,0,\ldots,0)}(x) = \left(\frac{2\arccosh x}{\sqrt{x^2-1}}\right)^d. \]
Thus, the only non-zero contribution in the limit comes from the term where $j= n$ and $\alpha=(d,0,\ldots0)$, and so
\begin{align*}
&\lim_{t\downarrow0} t^{d-1}e^{-\lambda^2/4t}\frac{\partial^n}{\partial r^n}Q_2^{(d-1)}(t,\cos(\sqrt{t}r)\cosh\lambda) \\
	&= 2 \lim_{t\downarrow0} e^{-\lambda^2/4t}\left(\frac{\partial^n}{\partial r^n}e^{\arccosh^2 (\cos(\sqrt{t}r)\cosh\lambda)/4t}\right)  
		 \left(\frac{\arccosh (\cos(\sqrt{t}r)\cosh\lambda)} 
		{2\sqrt{(\cos(\sqrt{t}r)\cosh\lambda)^2-1}}\right)^d.
\end{align*}
Again applying Lemma \ref{l.drq} and equation (\ref{e.easy1}) finishes the proof.
\end{proof}

In light of Proposition \ref{p.drq2}, the proof of (\ref{e.dzn2}) now works exactly as in the proof of Proposition \ref{p.dz} in the $SU(2)$ case. The following gives the proof of (\ref{e.drn2}).

\begin{prop}
\label{p.ddr}
Uniformly on compact subsets of $[0,\infty)\times\mathbb{R}$,
	\[\lim_{t\downarrow 0} t^{d+1}\frac{\partial^n}{\partial r^n} (p_{t,d}(\sqrt{t}r,tz))
	= 2 \frac{\partial^n}{\partial r^n} h_{1,d}(r,z).\]
\end{prop}

\begin{proof}
Again dealing with the derivatives of the Heisenberg kernel first, we note that equation (\ref{e.drh}) implies	
\[ \left|\left(\frac{\lambda}{\sinh\lambda }\right)^d\frac{\partial^n}{\partial r^n}e^{-r^2\lambda\coth\lambda/4}\right|
	\le C|\lambda|^{n+d}e^{-|\lambda|} \]
for all $r$ and sufficiently large $\lambda$. Thus, 
	\[\frac{\partial^n}{\partial r^n} h_{1,d}(r,z) 
	= \frac{1}{(4\pi)^{d+1}} \int_{-\infty}^\infty  e^{i \lambda z/2} 
		\left(\frac{\lambda}{\sinh\lambda }\right)^d\frac{\partial^n}{\partial r^n} e^{-r^2 \lambda \coth \lambda/4} d\lambda.\]

Now, let $K$ be a compact subset of $[0,\infty)\times\mathbb{R}$, and let $t>0$ be sufficiently small that $(\sqrt{t}r,tz)\in[0,\pi/4]\times[-\pi,\pi]$ for all $(r,z)\in K$.
Analogously to the proof of Proposition \ref{p.dr}, we may prove a generalization of Proposition \ref{p.G} and then use equation (\ref{e.qs}) to write
\begin{align*}
	t^{d+1}\frac{\partial^n}{\partial r^n}(p_{t,d}(\sqrt{t}r,tz))
	&= t^{d+1}\frac{e^{tz^2/4}}{\sqrt{4\pi t}}\int_\mathbb{R} e^{-\lambda^2/4t}e^{i\lambda z} \frac{\partial^n}{\partial r^n}q_{t,d}(\cos(\sqrt{t}r)\cosh\lambda)\,d\lambda \\
	&= t^{d-1}\frac{e^{d^2t}e^{tz^2/4}}{2(2\pi)^{d+1}}  
		 (I^n_{1,d}(t,r,z)+ I^n_{2,d}(t,r,z))   
\end{align*}
where
\begin{multline*}
	I_{2,d}^n(t,r,z):= \int_{\cos(\sqrt{t}r)\cosh\lambda>1} e^{-\lambda^2/4t}
		e^{i\lambda z/2} \\
	\times \frac{\partial^n}{\partial r^n}
	\Bigg\{\sum_{j=0}^{d-2} {d-1\choose j} Q_2^{(j)}(t,\cos(\sqrt{t}r)\cosh\lambda) R_2^{(d-1-j)}(t,\cos(\sqrt{t}r)\cosh\lambda) \\
	+ Q_2^{(d-1)}(t,\cos(\sqrt{t}r)\cosh\lambda))(1+ R_2(t,\cos(\sqrt{t}r)\cosh\lambda)\Bigg\} d\lambda 
\end{multline*}
	and $I_{1,d}^n$ is the companion integral over $\{\cos(\sqrt{t}r)\cosh\lambda<1\}$. 
	Again, we may show that $t^{d-1}e^{-\lambda^2/4t}\frac{\partial^n}{\partial r^n}Q_2^{(j)}R_2^{(d-1-j)}$ and $t^{d-1}e^{-\lambda^2/4t}\frac{\partial^n}{\partial r^n}Q_2^{(d-1)}(1+R_2)$ are uniformly bounded in $t$ by functions integrable in $\lambda$, and thus
we may move the limit in $t$ inside the integrand to yield the desired form. 
In particular, one may prove analogous estimates to Lemma \ref{l.Rr} to show that $\frac{\partial^m}{\partial r^m} R_2^{(\ell)}\rightarrow 0$ for $m,\ell\ge0$ as $t\downarrow 0$ and then use Proposition \ref{p.drq2} to show that for all $j=0,\ldots,d-2$
\begin{multline*}
	\lim_{t\downarrow0} t^{d-1}e^{-\lambda^2/4t}\\
		\times \frac{\partial^k}{\partial r^k}Q_2^{(j)}(t,(\cos(\sqrt{t}r)\cosh\lambda)) \frac{\partial^{n-k}}{\partial r^{n-k}}R_2^{(d-1-j)}(t,\cos(\sqrt{t}r)\cosh\lambda) = 0
\end{multline*}
for $k=0,\ldots,n$, and similarly these results imply that
	\begin{multline*}
	\lim_{t\downarrow0} t^{d-1}e^{-\lambda^2/4t} \frac{\partial^k}{\partial r^k}Q_2^{(d-1)}(t,(\cos(\sqrt{t}r)\cosh\lambda))\frac{\partial^{n-k}}{\partial r^{n-k}}(1+ R_2(t,\cos(\sqrt{t}r)\cosh\lambda) \\
		= \left\{ \begin{array}{ll} 0 & \text{if } k=0,\ldots,n-1 \\
		2\left(\frac{\lambda}{2\sinh\lambda}\right)^d \frac{\partial^n}{\partial r^n} e^{-r^2\lambda\coth\lambda/4} & \text{if } k=n \end{array}\right. .
\end{multline*}
The proof that $t^{d-1}I_{1,d}^n (t,r,z)\rightarrow0$ as $t\downarrow0$ works with similar (but easier) adaptations.
\end{proof}

\bibliographystyle{abbrv}

\end{document}